\documentclass[10pt]{article}
\pdfoutput=1 
\usepackage[top=1.3in,bottom=1.2in,left=1in,right=1in]{geometry}
\usepackage{amsmath}
\usepackage{amssymb}
\usepackage{amsthm}
\usepackage{mathtools}
\usepackage[only,llbracket,rrbracket]{stmaryrd}
\usepackage[scr=boondoxo]{mathalfa} 
\usepackage[english]{babel}
\usepackage{graphicx}
\usepackage[font=small,figurewithin=section]{caption}
\usepackage{subcaption}
\usepackage{float}
\usepackage[numbers]{natbib} 
\usepackage{hyperref}
\usepackage[all]{xy}
\usepackage{multirow}
\usepackage{array}
\usepackage{enumitem}
\usepackage{tikz}
\usetikzlibrary{matrix,arrows,decorations.pathmorphing}
\usepackage{color}
\usepackage{authblk}

\numberwithin{equation}{section}
\makeatletter
\g@addto@macro\bfseries{\boldmath}
\makeatother

\hypersetup{colorlinks=true,citecolor=black,filecolor=black,linkcolor=black,urlcolor=black}
\hypersetup{pdfstartview=FitB,pdfpagemode=UseNone}

\setlist{nolistsep}

\newcolumntype{L}[1]{>{\raggedright\let\newline\\\arraybackslash\hspace{0pt}}m{#1}}
\newcolumntype{C}[1]{>{\centering\let\newline\\\arraybackslash\hspace{0pt}}m{#1}}
\newcolumntype{R}[1]{>{\raggedleft\let\newline\\\arraybackslash\hspace{0pt}}m{#1}}
\newcolumntype{N}{@{}m{0pt}@{}}

\SetSymbolFont{stmry}{bold}{U}{stmry}{m}{n}

\newcommand{\R}{\mathbb{R}}

\newcommand{\C}{\mathbb{C}}

\newcommand{\Sym}{\mathbb{S}}
\newcommand{\Skw}{\mathbb{A}}
\newcommand{\Mat}{\mathbb{M}}

\newcommand{\mesh}{\mathcal{T}}

\newcommand{\dd}{\,\mathrm{d}}
\newcommand{\T}{\mathsf{T}}

\newcommand{\bdry}{\partial}
\newcommand{\tr}{\mathrm{tr}}

\newcommand{\opt}{\mathrm{opt}}
\newcommand{\enr}{\mathrm{enr}}

\DeclareMathOperator*{\argmin}{arg\,min}
\DeclareMathOperator{\grad}{grad}
\DeclareMathOperator{\curl}{curl}
\let\div\relax 
\DeclareMathOperator{\div}{div}

\makeatletter
\newcommand{\superimpose}[2]{%
  {\ooalign{$#1\@firstoftwo#2$\cr\hfil$#1\@secondoftwo#2$\hfil\cr}}}
\newcommand{\superimposecenter}[2]{%
  {\ooalign{$#1\@firstoftwo#2\vphantom{#2}$\cr\hidewidth$#1\@secondoftwo#2$\hidewidth\cr}}}
\makeatother

\newtheoremstyle{boldremark}
    {\dimexpr\topsep/2\relax} 
    {\dimexpr\topsep/2\relax} 
    {}          
    {}          
    {\bfseries} 
    {.}         
    {.5em}      
    {}          
\theoremstyle{definition}
\newtheorem*{definition*}{Definition}
\theoremstyle{plain}
\newtheorem{theorem}{Theorem}
\numberwithin{theorem}{section}
\newtheorem{lemma}{Lemma}
\numberwithin{lemma}{section}

\numberwithin{corollary}{section}
\theoremstyle{boldremark}
\newtheorem{remark}{Remark}
\numberwithin{remark}{section}

\makeatletter
\newcommand{\dashto}[1][2pt]{
	\settowidth{\@tempdima}{${}\rightarrow{}$}
	\makebox[\@tempdima]{${}\rightarrow{}$}
  \makebox[-\@tempdima]{\hspace{-0.1\@tempdima}\color{white}\rule[0.5ex]{#1}{1pt}}
  \makebox[\@tempdima]{}
	}
\makeatother
\let\tilde\widetilde

\DeclareRobustCommand{\widefrac}[3][5pt]{%
  \frac{\hspace{#1}#2\hspace{#1}}{\hspace{#1}#3\hspace{#1}}}
  
\newcommand{\upperone}[2]{\raisebox{-0.2\height}{$#1{}^1\!$}}
\newcommand{\upone}{\mathpalette\upperone\relax}
\newcommand{\lowertwo}[2]{\raisebox{-0.1\height}{$#1{}_{\!2}$}}
\newcommand{\lowtwo}{\mathpalette\lowertwo\relax}
\newcommand{\frachalf}[2]{\raisebox{0.0\height}{$#1{}\upone/\lowtwo$}}
\newcommand{\onehalf}{\mathpalette\frachalf\relax}

\newcommand{\scS}{\mathscr{S}}
\newcommand{\scU}{\mathscr{U}}
\newcommand{\scD}{\mathscr{D}}
\newcommand{\scM}{\mathscr{M}}
\newcommand{\scP}{\mathscr{P}}
\newcommand{\scC}{\mathscr{C}}
\newcommand{\scF}{\mathscr{F}}

\newcommand{\sfC}{\mathsf{C}}

\newcommand{\sfS}{\mathsf{S}}

\newcommand{\fku}{\mathfrak{u}}
\newcommand{\fkv}{\mathfrak{v}}

\newcommand{\pder}[2]{\frac{\partial #1}{\partial #2}}

\newcommand{\tder}[2]{\frac{\mathrm{d} #1}{\mathrm{d} #2}}

\newcommand{\nml}{n}

\begin{document}


\title{\vspace*{-35pt}Coupled variational formulations of linear elasticity and the DPG methodology}
\author[1]{Federico Fuentes}
\author[1]{Brendan Keith}
\author[1]{Leszek Demkowicz}
\author[2]{Patrick Le Tallec}
\affil[1]{The Institute for Computational Engineering and Sciences (ICES), The University of Texas at Austin, 201 E 24th St, Austin, TX 78712, USA}
\affil[2]{Laboratoire de M\'{e}chanique des Solides, \'{E}cole Polytechnique, Route de Saclay, 91128 Palaiseau, France}
\date{}


\maketitle

\vspace{-30pt}
\renewcommand{\abstractname}{\large Abstract}
\begin{abstract}
\small 
This article presents a general approach akin to domain-decomposition methods to solve a single linear PDE, but where each subdomain of a partitioned domain is associated to a distinct variational formulation coming from a mutually well-posed family of \textit{broken} variational formulations of the original PDE.
It can be exploited to solve challenging problems in a variety of physical scenarios where stability or a particular mode of convergence is desired in a part of the domain.
The linear elasticity equations are solved in this work, but the approach can be applied to other equations as well.
The broken variational formulations, which are essentially extensions of more standard formulations, are characterized by the presence of mesh-dependent broken test spaces and interface trial variables at the boundaries of the elements of the mesh.
This allows necessary information to be naturally transmitted between adjacent subdomains, resulting in \textit{coupled} variational formulations which are then proved to be globally well-posed.
They are solved numerically using the DPG methodology, which is especially crafted to produce stable discretizations of broken formulations.
Finally, expected convergence rates are verified in two different and illustrative examples.
\end{abstract}

\section{Introduction}
\label{sec:Introduction}

Many equations arising from physical applications can be given a variational formulation, which can then be analyzed using functional analysis, and solved using a discrete version of the formulation and the finite element method \cite{ciarlet2002finite,HughesFamous}.
In fact, for a fixed set of equations, there may even be multiple variational formulations which are essentially equivalent to the initial equations.
Typically, some of these formulations have significant advantages and disadvantages over the other formulations.
For example, in linear elasticity, the classical formulation coming from the principle of virtual work is well-known to be computationally efficient to solve with the Galerkin method, but is subject to volumetric locking phenomena for nearly incompressible materials \cite{HughesFamous}.
On the other hand, the Hellinger-Reissner mixed formulation is not as efficient, but it does remain robustly well-posed for nearly incompressible materials and produces a locally conservative stress tensor \cite{falk2008,brezzi2012mixed}.

This article studies the scenario where distinct variational formulations are implemented in different subdomains of the same physical domain.
This can be useful in situations where a certain behavior of the equations to be solved is known (or expected) in particular parts of the domain.
Hence, in each region one can choose a variational formulation which is well-suited to the expected behavior.
For example, consider a material with heterogeneous material properties varying within the domain.
The properties can vary continuously, as in cloaking applications or biological materials, or discontinuously, as in multi-material problems.
Then, in the parts of the domain where it can be an issue (e.g. a nearly incompressible material in linear elasticity), one can choose a variational formulation that is robustly well-posed with respect to the material properties.
In the remaining regions, where such robustness is not fundamental, one can choose a more computationally efficient formulation. 
Another example occurs when a near singularity is expected in a particular area, so that one would hope to use a variational formulation (with possibly an associated adaptivity scheme) which is desirable in that subdomain, but not necessarily in the entire physical domain \cite{DomainDecomMethods}.

The main issue with such an implementation arises at the interfaces between the two subdomains having distinct variational formulations.
At this interface, information must pass between the two subdomains to enable communication.
This imposes a coupling with both theoretical and practical compatibility issues which can be difficult to resolve and analyze. 
Moreover, the coupling must be constructed properly so that the entire problem is well-posed.
This is not immediate, even if each of the interacting variational formulations is well-posed when considered independently across the whole domain.

At the theoretical and infinite-dimensional level, an attractive possibility that naturally unburdens the compatibility and well-posedness requirements is the use of \textit{broken} variational formulations.
These mesh-dependent formulations are extensions of the usual variational formulations to the case involving broken (or discontinuous) test spaces.
They first arose in the analysis of the DPG methodology, which is a minimum residual method with broken test spaces \cite{demkowicz2010class,demkowicz2011class,DPGOverview}.
As will be seen, a family of distinct broken variational formulations can originate from the same system of equations through variational testing and by integrating by parts in different ways.
As expected, the formulations in the family will be closely related to each other.
In fact, for many systems of equations, the collection of formulations in each associated family can be shown to be mutually well-posed \cite{BrokenForms15,Keith2016,DemkowiczClosedRange}.
More importantly, the broken formulations in the family will be observed to inescapably possess interface variable unknowns which are a desirable means of communicating the necessary solution variable information across subdomains.
This is what allows introducing a proper definition of \textit{coupled} variational formulations, which will later be proved to be globally well-posed.

To actually compute approximate solutions to such a coupled system, one needs a method having discrete trial and test spaces that together retain well-posedness (i.e. numerical stability) \cite{babuska1971error,brezzi1974existence}.
This is easily achieved by use of the practical DPG methodology, the very method which motivated the systematic study of broken variational formulations.
Indeed, given a discrete trial space, the DPG methodology is especially crafted to approximate an optimal test space which reproduces the stability of the underlying infinite-dimensional broken variational formulation \cite{DPGOverview}.
Due to this unique design, it can be used with often discarded variational formulations which have different trial and test spaces, such as \textit{ultraweak} formulations \cite{DPGOverview,roberts2015discontinuous,bui2013unified}.
Apart from stability, the methodology carries other significant advantages including a well-behaved a posteriori error estimator for use in adaptive methods and a parallelizable assembly structure allowing local computation of optimal test functions from a standard (yet enriched) discretization of the underlying functional spaces.
However, the method is sometimes computationally intensive, because, when compared to standard methods, it typically comes at the cost of adding degrees of freedom along with some extra local computations \cite{roberts2014camellia}.

The purpose of this work is to demonstrate the use of the DPG methodology in solving the equations of linear elasticity via coupled variational formulations.
The family of variational formulations we study was introduced in \cite{Keith2016}, where all formulations were shown to be simultaneously well-posed.
Here, broken formulations will be shown to naturally dispose of many complications arising with the compatibility and well-posedness of coupled variational formulations.
Moreover, the use of the DPG methodology will corroborate expected theoretical convergence results.
Examples showing the viability of the approach at a practical level will be illustrated, including a case where the demanding scenario of a fully incompressible material is considered.
{This last case has physical applications in modeling steel braided rubber hoses and even stents.}

{
For the treatment of the DPG methodology as it applies to the equations of linear elasticity, it is worth highlighting \cite{Bramwell12,gopalakrishnan2014analysis,carstensen2014posteriori,Keith2016}.
}
{
Regarding the coupling of formulations, similarity exists between the approach in this work and that taken in \cite{HeuerKarkulikTransmission} (used for elliptic transmission problems).
There, a variational formulation similar to those considered here is coupled with a variational formulation composed of \textit{boundary integral} operators.
Afterward, the coupled formulation is discretized with the DPG methodology throughout the entire computational domain.
A remark is also warranted for the contributions in \cite{HeuerFuhrerCoupling,fuhrer2017coupling} where the ideas in \cite{HeuerKarkulikTransmission} are extended to couple the DPG methodology with more standard boundary element methods (BEMs), so that different discretization methods are considered across the domain.
}

This article is organized as follows.
In Section~\ref{sec:variational} a family of variational formulations equivalent to the equations of linear elasticity are introduced, followed by the associated family of broken variational formulations.
In Section~\ref{sec:coupled} coupled variational formulations are described.
The distinct broken formulations are shown to be compatible across common interfaces and the coupled formulations are proved to be well-posed.
In Section~\ref{sec:DPG} the DPG methodology used to solve the coupled formulations is outlined and novel linear algebra improvements are described.
Then, in Section~\ref{sec:results} two examples are exhibited, numerically solved, and discussed in detail.
The last example is a physically-relevant sheathed hose, for which a benchmark exact solution is derived.

\section{Variational formulations in linear elasticity}
\label{sec:variational}

\subsection{Linear elasticity equations}
\label{sec:linearelasticity}

In this work, the classical equations of static linear elasticity will be solved \cite{Ciarlet93}.
These are simply the linearization in the reference configuration about a stress-free state of the general constitutive model for solids and the conservation of momentum in the static case.
Indeed, let $\Omega\subseteq\R^3$ be a simply-connected bounded Lipschitz domain, whose boundary, $\bdry\Omega$, with normal, $\hat{\nml}$, can be partitioned into relatively open subsets, $\Gamma_u$ and $\Gamma_\sigma$, satisfying $\overline{\Gamma_u\cup\Gamma_\sigma}=\bdry\Omega$ and $\Gamma_u\cap\Gamma_\sigma=\varnothing$.
Then, the equations of linear elasticity can be written as the following first order system,
\begin{equation}
	\label{eq:LinElastSystem}
	\left\{
  \begin{aligned}
  	\sigma - \sfC\!:\!\varepsilon(u) &= 0\quad &&\text{in }\, \Omega\,,\\
   	-\div\sigma &= f\quad &&\text{in }\, \Omega\,,\\
   	u & = u^{\Gamma_u}\quad &&\text{on }\, \Gamma_u\,,\\
   	\sigma\!\cdot\!\hat{\nml} & = \sigma^{\Gamma_\sigma}_\nml\quad &&\text{on }\, \Gamma_\sigma\,,
  \end{aligned}
 	\right.
\end{equation}
where $u$ is the displacement, $\varepsilon(u)=\frac{1}{2}(\nabla u+\nabla u^\T)$ is its associated strain and $\sigma=\sigma^\T$ is the stress (which must be symmetric in order to satisfy conservation of angular momentum).
Meanwhile, $f$ is the known body force and the boundary conditions are given by the known displacement $u^{\Gamma_u}$ and traction $\sigma^{\Gamma_\sigma}_\nml$.
Lastly, $\sfC:\Sym\to\Sym$ is the stiffness tensor, where $\Sym$ is the space of symmetric $3\times3$ matrices.
For isotropic materials, it is $\sfC_{ijkl}=\lambda \delta_{ij}\delta_{kl}+\mu(\delta_{ik}\delta_{jl}+\delta_{il}\delta_{jk})$, where $\lambda$ and $\mu$ are the Lam\'e parameters.
All variables are assumed to be appropriately nondimensionalized.

The constitutive equation can be rewritten as 
\begin{equation}
	\sfS\!:\!\sigma-\varepsilon(u)=0\,,
	\label{eq:complianceconstitutive}
\end{equation}
where $\sfC^{-1}=\sfS:\Sym\to\Sym$ is the compliance tensor.
It is $\sfS_{ijkl}=\frac{1}{4\mu}(\delta_{ik}\delta_{jl}+\delta_{il}\delta_{jk})-\frac{\lambda}{2\mu(3\lambda+2\mu)}\delta_{ij}\delta_{kl}$ for isotropic materials.
This form is preferred when dealing with nearly incompressible materials (as $\lambda\to\infty$) because the norm of $\sfS$ remains finite, while that of $\sfC$ diverges.
This is the underlying reason why using this form in a variational setting prevents volumetric locking phenomena.

If instead of rewriting the constitutive equation one substitutes the stress $\sigma=\sfC\!:\!\varepsilon(u)$ into the equation for conservation of momentum, the resulting equation becomes $-\div(\sfC\!:\!\varepsilon(u))=f$.
This single second order equation with the corresponding boundary conditions is the starting point for the more traditional variational formulations.
However, as it will be seen, using the first order system gives more versatility to construct variational formulations.

\subsection{Hilbert spaces}
\label{sec:energyspaces}

Before proposing any variational formulations, it is first necessary to define the spaces of functions where the unknown variables will be sought.
First, define the $L^2$ inner product in a domain $K\subseteq\Omega$ as
\begin{equation}
	(u,v)_K=\int_K\tr(u^\T v)\dd K\,,
\end{equation}
where $\tr$ is the usual algebraic trace of a matrix, so that depending on whether $u$ and $v$ take scalar, vector or matrix values, $\tr(u^\T v)$ will be $uv$, $u\cdot v$ or $u:v$, respectively.
The underlying Hilbert spaces are
\begin{equation}
	\begin{aligned}
		L^2(K)&=\{u:K\to\R^3\mid \|u\|_{L^2(K)}^2=(u,u)_K<\infty\}\,,\\
		L^2(K;\mathbb{U})&=\{\omega:K\to\mathbb{U}\mid \|\omega\|_{L^2(K;\mathbb{U})}^2=(\omega,\omega)_K<\infty\}\,,\\
		H^1(K)&=\{u:K\to\R^3\mid\|u\|_{H^1(K)}^2=\|u\|_{L^2(K)}^2+\|\nabla u\|_{L^2(K;\Mat)}^2<\infty\}\,,\\
		H(\div,K)&=\{\sigma:K\to\Mat\mid\|\sigma\|_{H(\div,K)}^2=\|\sigma\|_{L^2(K;\Mat)}^2+\|\div\sigma\|_{L^2(K)}^2<\infty\}\,,\\
	\end{aligned}
	\label{eq:vectormatrixenergyspaces}
\end{equation}
where $\mathbb{U}$ is a subspace of $\Mat$, the space of $3\times3$ matrices, and where $\nabla u$ and $\div\sigma$ refer to the row-wise distributional gradient and divergence respectively.
More specifically, $\mathbb{U}$ can be the symmetric matrices, $\Sym$, the antisymmetric (or skew-symmetric) matrices, $\Skw$, or $\Mat$ itself.

Next, assume $K\subseteq\Omega$ is a Lipschitz domain and define the well-known surjective and continuous boundary trace operators,
\begin{equation}
\label{eq:ElementTrace}
	\begin{alignedat}{5}
		\tr_{\grad}^K&:H^1(K)\to H^{\onehalf}(\bdry K)\,,&&\qquad&\tr_{\grad}^K u&&&=u|_{\bdry K}\,,\\
		\tr_{\div}^K&:H(\div,K)\to H^{-\onehalf}(\bdry K)\,,&&\qquad&\tr_{\div}^K \sigma&&&=\sigma\!\cdot\!\hat{\nml}|_{\bdry K}\,,
	\end{alignedat}
\end{equation}
where $\hat{\nml}$ is the outward normal along $\bdry K$ and $\sigma\!\cdot\!\hat{\nml}$ represents a row-wise contraction.
Continuity of the operators is valid when $H^{\onehalf}(\bdry K)$ and $H^{-\onehalf}(\bdry K)$ are given minimum energy extension norms, in which case the two spaces are dual to each other.
Indeed, the following equalities were proved in \cite{BrokenForms15},
\begin{equation}
	\begin{aligned}
		\|\hat{u}\|_{H^{\onehalf}(\bdry K)}
			&=\inf_{u\in(\tr_{\grad}^K)^{-1}\{\hat{u}\}}\|u\|_{H^1(K)}
				=\sup_{\sigma\in H(\div,K)\setminus\{0\}}\widefrac[-12pt]{|\langle\hat{u},\tr_{\div}^K\sigma\rangle_{\bdry K}|}
					{\qquad\quad\|\sigma\|_{H(\div,K)}}\quad\,,\\
		\|\hat{\sigma}_\nml\|_{H^{-\onehalf}(\bdry K)}
			&=\inf_{\sigma\in(\tr_{\div}^K)^{-1}\{\hat{\sigma}_\nml\}}\|\sigma\|_{H(\div,K)}
				=\sup_{u\in H^1(K)\setminus\{0\}}\widefrac[0pt]{|\langle\hat{\sigma}_\nml,\tr_{\grad}^Ku\rangle_{\bdry K}|}
					{\qquad\|u\|_{H^1(K)}}\,,
	\end{aligned}
	\label{eq:traceidentitiesK}
\end{equation}
where $\langle\cdot,\cdot\rangle_{\bdry K}$ is the duality pairing $\langle\cdot,\cdot\rangle_{H^{\onehalf}(\bdry K)\times H^{-\onehalf}(\bdry K)}$ or $\langle\cdot,\cdot\rangle_{H^{-\onehalf}(\bdry K)\times H^{\onehalf}(\bdry K)}$ depending upon the context.
The duality pairings are well-defined by $\langle\hat{u},\hat{\sigma}_\nml\rangle_{H^{\onehalf}(\bdry K)\times H^{-\onehalf}(\bdry K)}=(u,\div\sigma)_K+(\nabla u,\sigma)_K$ for any $u\in(\tr_{\grad}^K)^{-1}\{\hat{u}\}$ and $\sigma\in(\tr_{\div}^K)^{-1}\{\hat{\sigma}_\nml\}$.
When $K=\Omega$, recall $\bdry\Omega$ is partitioned into $\Gamma_u$ and $\Gamma_\sigma$, and it may be useful to define spaces that vanish in these parts of the boundary.
Hence, define
\begin{equation}
	\begin{aligned}
		H_{\Gamma_u}^1(\Omega)&=\{u\in H^1(\Omega)\mid \tr_{\grad}^\Omega u|_{\Gamma_u}=0\}\,,\\
		H_{\Gamma_\sigma}(\div,\Omega)&=\{\sigma\in H(\div,\Omega)\mid\tr_{\div}^\Omega \sigma|_{\Gamma_\sigma}=0\}\,,
	\end{aligned}
\end{equation}
where the distributions $\tr_{\grad}^\Omega u|_{\Gamma_u}$ and $\tr_{\div}^\Omega \sigma|_{\Gamma_\sigma}$ are defined by restriction.

Usually, the spaces defined above (for $K=\Omega$) are all that is required, but, as it will be seen, \textit{broken} spaces will also be needed.
To define them, one must first note that these spaces are broken along a given mesh (i.e. an open partition), $\mesh$, of $\Omega$, containing elements (i.e. subdomains) $K\in\mesh$.
Hence, these spaces are mesh-dependent.
Indeed, they are merely those spaces in $L^2$ in $\Omega$ which are elementwise (locally defined by restriction) in the desired underlying space,
\begin{equation}
	\begin{aligned}
		L^2(\mesh)&=\{u\in L^2(\Omega)\mid\|u\|_{L^2(\mesh)}^2=\textstyle{\sum_{K\in\mesh}}\|u|_K\|_{L^2(K)}^2<\infty\}=L^2(\Omega)\,,\\
		L^2(\mesh;\mathbb{U})&=\{\omega\in L^2(\Omega;\mathbb{U})\mid\|\omega\|_{L^2(\mesh;\mathbb{U})}^2
			=\textstyle{\sum_{K\in\mesh}}\|\omega|_K\|_{L^2(K;\mathbb{U})}^2<\infty\}=L^2(\Omega;\mathbb{U})\,,\\
		H^1(\mesh)&=\{u\in L^2(\Omega)\mid\|u\|_{H^1(\mesh)}^2=\textstyle{\sum_{K\in\mesh}}\|u|_K\|_{H^1(K)}^2<\infty\}\,,\\
		H(\div,\mesh)&=\{\sigma\in L^2(\Omega;\Mat)\mid
			\|\sigma\|_{H(\div,\mesh)}^2=\textstyle{\sum_{K\in\mesh}}\|\sigma|_K\|_{H(\div,K)}^2<\infty\}\,.
	\end{aligned}
	\label{eq:BrokenSpacesAndNorms}
\end{equation}
Notice $H^1(\Omega)\subseteq H^1(\mesh)$ and $H(\div,\Omega)\subseteq H(\div,\mesh)$ because in general $H^1(\mesh)$ and $H(\div,\mesh)$ may have elements whose gradient and divergence over $\Omega$ are singular distributions (due to the presence of singularities at the the boundaries of the elements $K\in\mesh$), yet the restriction of the gradient and divergence to every $K\in\mesh$ is a regular distribution in $L^2(K;\Mat)$ and $L^2(K)$ respectively.
The $L^2$ mesh inner product associated to $\mesh$ is,
\begin{equation}
	(u,v)_\mesh=\sum_{K\in\mesh}(u|_K,v|_K)_K\,.
\end{equation}
Note that $(\cdot,\cdot)_\mesh$ is applicable even to singular distributions whose restriction to every $K\in\mesh$ is a regular distribution.
Thus, writing $(u,u)_\Omega=(u,u)_\mesh$ is only valid when $u\in L^2(\Omega)$ (or $L^2(\Omega;\mathbb{U})$).
Similarly, $\|u\|_{H^1(\mesh)}^2=(u,u)_\mesh+(\nabla u,\nabla u)_\mesh=(u,u)_\Omega+(\nabla u,\nabla u)_\Omega=\|u\|_{H^1(\Omega)}^2$ is valid when $u\in H^1(\Omega)$, but not when $u\in H^1(\mesh)$ because $\nabla u$ may not lie in $L^2(\Omega;\Mat)$.

Functions in broken spaces essentially have information isolated within each element of the mesh since there is technically no notion of a shared trace along adjacent elements.
Hence, broken spaces are usually accompanied by interface spaces that are intended to somehow share information along adjacent elements of the mesh.
Thus, interface spaces are also mesh-dependent.
To define them, first consider the mesh boundary trace operators
\begin{equation}
	\begin{alignedat}{5}
		\tr_{\grad}&:H^1(\mesh)\to\prod_{K\in\mesh}H^{\onehalf}(\bdry K)=H_{\Pi}^{\onehalf}(\bdry\mesh)\,,&&\qquad
			&\tr_{\grad} u&&&=\prod_{K\in\mesh}\tr_{\grad}^K u|_K\,,\\
		\tr_{\div}&:H(\div,\mesh)\to\prod_{K\in\mesh}H^{-\onehalf}(\bdry K)=H_{\Pi}^{-\onehalf}(\bdry\mesh)\,,&&\qquad
			&\tr_{\div}\sigma&&&=\prod_{K\in\mesh}\tr_{\div}^K \sigma|_K\,,
	\end{alignedat}
\end{equation}
where $H_{\Pi}^{\onehalf}(\bdry\mesh)$ and $H_{\Pi}^{-\onehalf}(\bdry\mesh)$ are endowed with the standard Hilbert norm for product spaces and are dual to each other,
\begin{equation}
	\begin{gathered}
		\|\hat{u}\|_{H_{\Pi}^{\onehalf}(\bdry\mesh)}^2=\sum_{K\in\mesh}\|\hat{u}_K\|_{ H^{\onehalf}(\bdry K)}^2\,,\qquad\quad
			\|\hat{\sigma}_\nml\|_{H_{\Pi}^{-\onehalf}(\bdry \mesh)}^2
				=\sum_{K\in\mesh}\|\hat{\sigma}_{\nml,K}\|_{ H^{-\onehalf}(\bdry K)}^2\,,\\
		\langle\cdot,\cdot\rangle_{\bdry\mesh}=\sum_{K\in\mesh}\langle\cdot,\cdot\rangle_{\bdry K}\,.
	\end{gathered}
	\label{eq:BrokenTraceNorms}
\end{equation}
The following equalities hold for all $\hat{u}\in H_{\Pi}^{\onehalf}(\bdry\mesh)$ and $\hat{\sigma}_\nml\in H_{\Pi}^{-\onehalf}(\bdry\mesh)$ (see \eqref{eq:traceidentitiesK} and \cite{BrokenForms15}),
\begin{equation}
	\begin{aligned}
		\|\hat{u}\|_{H_{\Pi}^{\onehalf}(\bdry\mesh)}
				&=\inf_{u\in\tr_{\grad}^{-1}\{\hat{u}\}}\|u\|_{H^1(\mesh)}
					=\sup_{\sigma\in H(\div,\mesh)\setminus\{0\}}\widefrac[-12pt]{|\langle\hat{u},\tr_{\div}\sigma\rangle_{\bdry\mesh}|}
						{\qquad\quad\|\sigma\|_{H(\div,\mesh)}}\quad\,,\\
		\|\hat{\sigma}_\nml\|_{H_{\Pi}^{-\onehalf}(\bdry\mesh)}
				&=\inf_{\sigma\in\tr_{\div}^{-1}\{\hat{\sigma}_\nml\}}\|\sigma\|_{H(\div,\mesh)}
					=\sup_{u\in H^1(\mesh)\setminus\{0\}}\widefrac[0pt]{|\langle\hat{\sigma}_\nml,\tr_{\grad}u\rangle_{\bdry\mesh}|}
						{\qquad\|u\|_{H^1(\mesh)}}\,.
	\end{aligned}
	\label{eq:traceidentitiesmesh}
\end{equation}
Now, the actual interface spaces that are of interest are only closed subspaces of $H_{\Pi}^{\onehalf}(\bdry\mesh)$ and $H_{\Pi}^{-\onehalf}(\bdry\mesh)$,
\begin{equation}
	H_{\Gamma_u}^{\onehalf}(\bdry\mesh)=\tr_{\grad}(H_{\Gamma_u}^1(\Omega))\,,\qquad\qquad
		H_{\Gamma_\sigma}^{-\onehalf}(\bdry\mesh)=\tr_{\div}(H_{\Gamma_\sigma}(\div,\Omega))\,,
\end{equation}
and they inherit the corresponding norms from $H_{\Pi}^{\onehalf}(\bdry\mesh)$ and $H_{\Pi}^{-\onehalf}(\bdry\mesh)$, respectively.

\subsection{A family of variational formulations}

A linear variational formulation is a problem of the form, 
\begin{equation}
	\label{eq:bilinearFormEQ}
 	\left\{
  	\begin{aligned}
   		&\text{Find } \fku\in U\,,\\
   		&b(\fku,\fkv) = \ell(\fkv)\,,\quad \text{for all }\, \fkv\in V\,,
  	\end{aligned}
 	\right.
\end{equation}
where $U$ and $V$ are trial and test Hilbert spaces over a fixed field $\mathbb{F}\in\{\R,\mathbb{C}\}$, $b:U\times V \to \mathbb{F}$ is a continuous bilinear form if $\mathbb{F}=\R$ or sesquilinear form if $\mathbb{F}=\C$, and $\ell\in V'$ is a continuous linear form if $\mathbb{F}=\R$ or antilinear form if $\mathbb{F}=\C$.

The most common way of constructing variational formulations which embody an underlying equation is by formally multiplying the equation by a test function and integrating.
One can then choose whether or not to formally integrate by parts; a process which, once the trial and test spaces are chosen appropriately, technically leads to different variational formulations of the same original equation. 
Applying the simplest form of this process to the first order system of linear elasticity in \eqref{eq:LinElastSystem} leads to four different variational formulations, which naturally involve the space $H(\div,K;\Sym)=\{\sigma:K\to\Sym\mid\sigma\in H(\div,K)\}$ \cite{Keith2016}.
This space has a strongly enforced symmetry of the second order tensors in the definition of the space itself.
However, $H(\div,K;\Sym)$ is notoriously difficult to discretize as a high order space with mathematically desirable properties \cite{elas3dfamily,mixedelas3d,JunHuStrongSymmetry}.
Thus, one common approach is to use $H(\div,K)$ instead and impose the necessary symmetries weakly through auxiliary variables, such as in \cite{Keith2016}, where the resulting bilinear forms of the four variational formulations are,
\begin{align}
	&\begin{aligned}
		&U^\scS=H_{\Gamma_\sigma}(\div,\Omega)\times H^1_{\Gamma_u}(\Omega)\,,\qquad\qquad
			V^\scS=L^2(\Omega;\Sym)\times L^2(\Omega)\times L^2(\Omega;\Skw)\,,\\
		&b^\scS((\sigma,u),(\tau,v,w))=(\sigma,\tau)_\Omega-(\sfC:\nabla u,\tau)_\Omega-(\div\sigma,v)_\Omega+(\sigma,w)_\Omega\,,
		\label{eq:VarFormStrong}
	\end{aligned}\displaybreak[2]\\[2mm]
	&\begin{aligned}
		&U^\scU=L^2(\Omega;\Sym)\times L^2(\Omega)\times L^2(\Omega;\Skw)\,,\qquad\qquad
			V^\scU=H_{\Gamma_\sigma}(\div,\Omega)\times H^1_{\Gamma_u}(\Omega)\,,\\
		&b^\scU((\sigma,u,\omega),(\tau,v))=(\sfS:\sigma,\tau)_\Omega+(\omega,\tau)_\Omega+(u,\div\tau)_\Omega+(\sigma,\nabla v)_\Omega\,,
		\label{eq:VarFormUltraweak}
	\end{aligned}\displaybreak[2]\\[2mm]
	&\begin{aligned}
		&U^\scD=L^2(\Omega;\Sym)\times H^1_{\Gamma_u}(\Omega)\,,\qquad\qquad
			V^\scD=L^2(\Omega;\Sym)\times H^1_{\Gamma_u}(\Omega)\,,\\
		&b^\scD((\sigma,u),(\tau,v))=(\sigma,\tau)_\Omega-(\sfC:\nabla u,\tau)_\Omega+(\sigma,\nabla v)_\Omega\,,
		\label{eq:VarFormDualmixed}
	\end{aligned}\displaybreak[2]\\[2mm]
	&\begin{aligned}
		&U^\scM=H_{\Gamma_\sigma}(\div,\Omega)\times L^2(\Omega)\times L^2(\Omega;\Skw)\,,\qquad\qquad
			V^\scM=H_{\Gamma_\sigma}(\div,\Omega)\times L^2(\Omega)\times L^2(\Omega;\Skw)\,,\\
		&b^\scM((\sigma,u,\omega),(\tau,v,w))=(\sfS:\sigma,\tau)_\Omega+(\omega,\tau)_\Omega+(u,\div\tau)_\Omega-(\div\sigma,v)_\Omega
			+(\sigma,w)_\Omega\,.
		\label{eq:VarFormMixed}
	\end{aligned}
	\intertext{Here, $\scS$ stands for strong, $\scU$ stands for ultraweak, $\scD$ stands for dual-mixed and $\scM$ stands for mixed.
Moreover, taking the second order form of the equation as a starting point (instead of the first order system), one can also obtain the more well-known primal variational formulation, $\scP$,
	}
	&\begin{aligned}
		&U^\scP=H^1_{\Gamma_u}(\Omega)\,,\qquad\qquad
			V^\scP=H^1_{\Gamma_u}(\Omega)\,,\\
		&b^\scP(u,v)=(\sfC:\nabla u,\nabla v)_\Omega\,.
		\label{eq:VarFormPrimal}
	\end{aligned}	
\end{align}
With homogeneous boundary conditions, $u^{\Gamma_u}=0$ and $\sigma^{\Gamma_\sigma}_\nml=0$, the linear forms, $\ell^{\scF}$, always take the form $\ell^{\scF}(\fkv)=(f,v)_\Omega$, where $v$ is a component of $\fkv\in V^{\scF}$ with $\scF$ being one of the formulations defined above.
With nonhomogeneous boundary conditions, $\ell^{\scF}$ will have terms involving extensions of the boundary conditions, $u^{\Gamma_u}$ and $\sigma^{\Gamma_\sigma}_\nml$, to $\bdry\Omega$ and $\Omega$.
For example, $\ell^\scU((\tau,v))=(f,v)_\Omega+\langle\check{u}^{\Gamma_u},\tr_{\div}^\Omega\tau\rangle_{\bdry\Omega}+\langle\check{\sigma}^{\Gamma_\sigma}_\nml,\tr_{\grad}^\Omega v\rangle_{\bdry\Omega}$, where $\check{u}^{\Gamma_u}$ and $\check{\sigma}^{\Gamma_\sigma}_\nml$ are some extension to $\bdry\Omega$ of $u^{\Gamma_u}$ and $\sigma^{\Gamma_\sigma}_\nml$ respectively.
As these expressions suggest, it is assumed that $f\in L^2(\Omega)$, $u^{\Gamma_u}\in\tr_{\grad}^\Omega(H^1(\Omega))|_{\Gamma_u}$ and $\sigma^{\Gamma_\sigma}_\nml\in\tr_{\div}^\Omega(H(\div,\Omega))|_{\Gamma_\sigma}$.
Moreover, whenever necessary, $\sfC$ and $\sfS$ are assumed to act on $\Mat$ (as opposed to merely $\Sym$) via the trivial extensions $\sfC|_\Skw=0$ and $\sfS|_\Skw=0$ (see \cite{Keith2016}).

In \cite{Keith2016} it was proved that, provided $\Gamma_u\neq\varnothing$, all the previously defined variational formulations (along with those making use of $H(\div,K;\Sym)$) are simultaneously well-posed in the sense of Hadamard.
That is, for the problem \eqref{eq:bilinearFormEQ} with the forms and spaces coming from one of \eqref{eq:VarFormStrong}--\eqref{eq:VarFormPrimal}, there exists a unique solution $\fku^\scF\in U^\scF$ satisfying the stability estimate $\|\fku^\scF\|_{U^\scF}\leq\frac{1}{\gamma^\scF}\|\ell\|_{(V^\scF)'}$ for some $\gamma^\scF>0$.
Since all variational formulations originate in the same equations, by testing with smooth functions it is made clear that the unique solutions agree among all formulations.

It should be noted that the list of variational formulations for the equations of linear elasticity proposed here is by no means exhaustive.
Indeed, alternative versions of \eqref{eq:VarFormStrong} and \eqref{eq:VarFormDualmixed} containing the compliance tensor (via use of \eqref{eq:complianceconstitutive}) are possible to construct, while in \cite{Oden76} energy functionals are used to propose up to fourteen different variational formulations for these equations.

\subsection{Broken variational formulations}
\label{sub:BrokenVariationFormulations}

For some numerical methods, mesh-dependent broken spaces can bring advantages.
In particular, consider the case where only the \textit{test} spaces are broken.
It is in this setting that broken variational formulations arise and, as it will be seen, this is fundamental in order to localize certain computations in the DPG methodology.

Consider a mesh, $\mesh$, of $\Omega$, containing elements $K\in\mesh$.
Instead of following the original approach of formally multiplying by a test function, integrating over $\Omega$, and integrating by parts if desired, the idea in this case is to integrate over each $K\in\mesh$ and then sum the contributions.
This differs from the former scenario in that the test functions can now be \textit{broken}, so that they may have trace discontinuities along the boundaries of adjacent elements in the mesh.
Thus, when integration by parts is performed, some mesh boundary terms seize to cancel and have to be explicitly considered.
Apart from these terms, the resulting formulations are the same as before, where \textit{unbroken} test functions were being used.
However, if they are to retain as much mathematical structure from the original \textit{unbroken} variational formulations, one finds that the new mesh boundary terms must have a life of their own and become additional independent variables.
That is, the price of using broken test functions is that one sometimes needs to define new mesh-dependent \textit{interface} variables along the boundary of the mesh (see \cite{DPGOverview,BrokenForms15,Keith2016} for more details).

\textit{Broken variational formulations} are precisely those formulations with broken test spaces constructed as described above.
They are clearly related to the original unbroken variational formulations, which do not require the test spaces to be broken.
In fact, the bilinear forms of broken variational formulations can be decoupled into two bilinear forms as,
\begin{equation}
	b(\fku,\fkv)= b_0(\fku_0,\fkv)+\hat{b}(\hat{\fku},\fkv)\,,
	\label{eq:BrokBilinDecouple}
\end{equation}
where $\fku=(\fku_0,\hat{\fku})\in U=U_0\times\hat{U}$ and $\fkv\in V$, with $U_0$ being the space associated to the original unbroken formulation, $\hat{U}$ being a space of interface variables, and $V$ being the broken test space directly associated to the test space $V_0\subseteq V$ coming from the original unbroken formulation.
When the test space is restricted from $V$ to $V_0$ the variational formulation collapses to the original unbroken formulation.
More precisely, $b_0|_{U_0\times V_0}$ is the bilinear form from the unbroken formulation and $\hat{b}|_{\hat{U}\times V_0}=0$, while $\ell|_{V_0}$ is the linear form from the unbroken formulation.
In this sense, a broken variational formulation can be interpreted as an extension to broken test spaces of an unbroken variational formulation.
It can be shown that the well-posedness of broken variational formulations depends on the well-posedness of the original unbroken variational formulation and that of $\hat{b}$ (see \cite{BrokenForms15}).
Moreover, the unique solution $\fku_0\in U_0$ to the unbroken formulation is the $U_0$ component of the unique solution $(\fku_0,\hat{\fku})\in U_0\times\hat{U}$ to the broken variational formulation.

The broken variational formulations associated to \eqref{eq:VarFormStrong}--\eqref{eq:VarFormPrimal} are
\begin{align}
	&\begin{aligned}
		&U_0^{\scS_\mesh}=U^\scS\,,\quad \hat{U}^{\scS_\mesh}=\varnothing\,,\qquad
			V^{\scS_\mesh}=L^2(\mesh;\Sym)\times L^2(\mesh)\times L^2(\mesh;\Sym)\,,\\
		&b_0^{\scS_\mesh}((\sigma,u),(\tau,v,w))=(\sigma,\tau)_\mesh-(\sfC:\nabla u,\tau)_\mesh
			-(\div\sigma,v)_\mesh+(\sigma,w)_\mesh\,,
		\label{eq:BrokVarFormStrong}
	\end{aligned}\displaybreak[2]\\[2mm]
	&\begin{aligned}
		&U_0^{\scU_\mesh}=U^\scU\,,\quad 
			\hat{U}^{\scU_\mesh}=H_{\Gamma_u}^{\onehalf}(\bdry\mesh) \times H_{\Gamma_\sigma}^{-\onehalf}(\bdry\mesh)\,,\qquad
				V^{\scU_\mesh}=H(\div,\mesh)\times H^1(\mesh)\,,\\
		&b_0^{\scU_\mesh}((\sigma,u,\omega),(\tau,v))
			=(\sfS:\sigma,\tau)_\mesh+(\omega,\tau)_\mesh+(u,\div\tau)_\mesh+(\sigma,\nabla v)_\mesh\,,\\
		&\hat{b}^{\scU_\mesh}((\hat{u},\hat{\sigma}_\nml),(\tau,v))=
			-\langle\hat{u},\tr_{\div}\tau\rangle_{\bdry\mesh}-\langle\hat{\sigma}_\nml,\tr_{\grad}v\rangle_{\bdry\mesh}\,,
		\label{eq:BrokVarFormUltraweak}
	\end{aligned}\displaybreak[2]\\[2mm]
	&\begin{aligned}
		&U_0^{\scD_\mesh}=U^\scD\,,\quad 
			\hat{U}^{\scD_\mesh}=H_{\Gamma_\sigma}^{-\onehalf}(\bdry\mesh)\,,\qquad
				V^{\scD_\mesh}=L^2(\mesh;\Sym)\times H^1(\mesh)\,,\\
		&b_0^{\scD_\mesh}((\sigma,u),(\tau,v))=(\sigma,\tau)_\mesh-(\sfC:\nabla u,\tau)_\mesh+(\sigma,\nabla v)_\mesh\,,\\
		&\hat{b}^{\scD_\mesh}(\hat{\sigma}_\nml,(\tau,v))=-\langle\hat{\sigma}_\nml,\tr_{\grad}v\rangle_{\bdry\mesh}\,,
		\label{eq:BrokVarFormDualmixed}
	\end{aligned}\displaybreak[2]\\[2mm]
	&\begin{aligned}
		&U_0^{\scM_\mesh}=U^\scM\,,\quad 
			\hat{U}^{\scM_\mesh}=H_{\Gamma_u}^{\onehalf}(\bdry\mesh)\,,\qquad
				V^{\scM_\mesh}=H(\div,\mesh)\times L^2(\mesh)\times L^2(\mesh;\Skw)\,,\\
		&b_0^{\scM_\mesh}((\sigma,u,\omega),(\tau,v,w))=(\sfS:\sigma,\tau)_\mesh+(\omega,\tau)_\mesh+(u,\div\tau)_\mesh-(\div\sigma,v)_\mesh
			+(\sigma,w)_\mesh\,,\\
		&\hat{b}^{\scM_\mesh}(\hat{u},(\tau,v,w))=-\langle\hat{u},\tr_{\div}\tau\rangle_{\bdry\mesh}\,,
		\label{eq:BrokVarFormMixed}
	\end{aligned}\displaybreak[2]\\[2mm]
	&\begin{aligned}
		&U_0^{\scP_\mesh}=U^\scP\,,\quad
			\hat{U}^{\scP_\mesh}=H_{\Gamma_\sigma}^{-\onehalf}(\bdry\mesh)\,,\qquad
				V^{\scP_\mesh}=H^1(\mesh)\,,\\
		&b_0^{\scP_\mesh}(u,v)=(\sfC:\nabla u,\nabla v)_\mesh\,,\\
		&\hat{b}^{\scP_\mesh}(\hat{\sigma}_\nml,v)=-\langle\hat{\sigma}_\nml,\tr_{\grad}v\rangle_{\bdry\mesh}\,,
		\label{eq:BrokVarFormPrimal}
	\end{aligned}
\end{align}
where $U^{\scF_\mesh}=U_0^{\scF_\mesh}\times\hat{U}^{\scF_\mesh}$ with $\scF$ being a placeholder for one of the preceding formulations, and where $b^{\scF_\mesh}:U^{\scF_\mesh}\times V^{\scF_\mesh}\to\R$ is defined in terms of $b_0^{\scF_\mesh}$ and $\hat{b}^{\scF_\mesh}$ by \eqref{eq:BrokBilinDecouple}.
As before, the linear forms $\ell^{\scF_\mesh}$ always have the term $(f,v)_\mesh$ and additionally may include terms involving extensions of the boundary conditions $u^{\Gamma_u}$ and $\sigma^{\Gamma_\sigma}_\nml$, to $\Omega$ and the boundary of the mesh (by use of $\tr_{\grad}$ and $\tr_{\div}$ on an extension to $\Omega$).
For example, $\ell^{\scU_\mesh}((\tau,v))=(f,v)_\mesh+\langle\breve{u}^{\Gamma_u},\tr_{\div}\tau\rangle_{\bdry\mesh}+\langle\breve{\sigma}^{\Gamma_\sigma}_\nml,\tr_{\grad} v\rangle_{\bdry\mesh}$, where $\breve{u}^{\Gamma_u}$ and $\breve{\sigma}^{\Gamma_\sigma}_\nml$ are some extension to $\tr_{\grad}(H^1(\Omega))$ and $\tr_{\div}(H(\div,\Omega))$ of $u^{\Gamma_u}$ and $\sigma^{\Gamma_\sigma}_\nml$ respectively.
As expected, $b^{\scF_\mesh}$ and $\ell^{\scF_\mesh}$ can be viewed as extensions to the original forms $b^{\scF}$ and $\ell^{\scF}$, because they collapse to the latter when testing against unbroken test functions in $V^{\scF}\subseteq V^{\scF_\mesh}$.
That is, $b_0^{\scF_\mesh}|_{U^{\scF}\times V^{\scF}}=b^{\scF}$, $\hat{b}^{\scF_\mesh}|_{\hat{U}^{\scF_\mesh}\times V^{\scF}}=0$ and $\ell^{\scF_\mesh}|_{V^{\scF}}=\ell^{\scF}$.

As long as $\Gamma_u\neq\varnothing$, the broken variational formulations in \eqref{eq:BrokVarFormStrong}--\eqref{eq:BrokVarFormPrimal} were shown to be mutually well-posed in \cite{Keith2016} by using fundamental results proved in \cite[Theorem~3.1]{BrokenForms15} along with the proof of mutual well-posedness of the original formulations in \eqref{eq:VarFormStrong}--\eqref{eq:VarFormPrimal} (proved in \cite{Keith2016}) and the identities in \eqref{eq:traceidentitiesmesh} (proved in \cite{BrokenForms15}).
The constant associated to the stability estimate from the well-posedness is independent from the choice of the mesh.

\section{Coupled variational formulations}
\label{sec:coupled}

As mentioned initially, there are multiple reasons that explain why it is desirable to solve the equations of linear elasticity with different variational formulations on distinct subdomains of the initial domain.
The challenge in attaining this goal is that one must find a way of communicating solution information across the shared boundaries of the subdomains.
For the purpose of illustration, simply consider a domain $\Omega$ partitioned into two disjoint subdomains, $\Omega^\scU$ and $\Omega^\scP$, with a common interface, $\Gamma_I$, such that $\overline{\Omega}{}^\scU\cup\overline{\Omega}{}^\scP=\overline{\Omega}$ and $\overline{\Omega}{}^\scU\cap\overline{\Omega}{}^\scP=\overline{\Gamma}_I$.
As suggested by the notation, suppose that the equations of linear elasticity are to be solved in $\Omega^\scU$ via the ultraweak variational formulation in \eqref{eq:VarFormUltraweak}, and in $\Omega^\scP$ via the primal variational formulation in \eqref{eq:VarFormPrimal}.
If a solution is to exist, then it should be compatible in some sense at the common interface $\Gamma_I$. 
This immediately poses a theoretical concern because the displacement variable in the ultraweak variational formulation lies in $L^2(\Omega)|_{\Omega^\scU}$ and so it does not even have a notion of trace at $\Gamma_I$.
Thus, it is not compatible with the primal displacement variable which lies in $H_{\Gamma_u}^1(\Omega)|_{\Omega^\scP}$.
A similar issue also arises with the test spaces, which are obviously different on each subdomain.
Even though the finite-dimensional trial and test subspaces of any naive discretization generally do have well-defined traces, these difficulties are reasonably expected to be inherited by the discretization, meaning any discrete convergence or stability analysis will probably be laborious, if at all possible.
Hence, the goal is to resolve the compatibility concerns at the infinite-dimensional level by developing a globally well-posed variational problem. 
Once this is done, there will be a clearer hope of producing stable and convergent discretizations.

The claim is that by using broken variational formulations, the theoretical compatibility issues are naturally dealt with.
To see this, suppose instead that the equations are to be solved in $\Omega^\scU$ with the ultraweak \textit{broken} variational formulation in \eqref{eq:BrokVarFormUltraweak} and in $\Omega^\scP$ with the primal \textit{broken} variational formulation in \eqref{eq:BrokVarFormPrimal}.
The mesh associated to the broken formulations, $\mesh$, is obviously assumed to be consistent with the subdomain partitioning, meaning that it is a refinement of the subdomain mesh, $\mesh_0=\{\Omega^\scU,\Omega^\scP\}$, and as such, there exist submeshes $\mesh^\scU$ and $\mesh^\scP$ of $\Omega^\scU$ and $\Omega^\scP$ respectively, such that $\mesh=\mesh^\scU\cup\mesh^\scP$.
In this scenario, the displacement variable in the ultraweak domain still lies in $L^2(\Omega)|_{\Omega^\scU}$, but the difference is that now there is an extra \textit{interface} displacement variable, $\hat{u}^\scU\in H_{\Gamma_u}^{\onehalf}(\bdry\mesh)|_{\mesh^\scU}$.
This variable is very convenient, as it is theoretically compatible at $\Gamma_I$ with the well-defined trace of the displacement variable of the primal variational formulation, $u^\scP\in H^1_{\Gamma_u}(\Omega)|_{\Omega^\scP}$.
Similarly, with regard to the stress, there exist two new interface traction variables, $\hat{\sigma}_\nml^\scU\in H_{\Gamma_\sigma}^{-\onehalf}(\bdry\mesh)|_{\mesh^\scU}$ and $\hat{\sigma}_\nml^\scP\in H_{\Gamma_\sigma}^{-\onehalf}(\bdry\mesh)|_{\mesh^\scP}$, which are naturally compatible at $\Gamma_I$.
Meanwhile, the use of broken test spaces relinquishes any compatibility requirements at the level of test spaces.
Notice the compatibility is not limited to the broken ultraweak and primal formulations.
Indeed, a close observation of the broken variational formulations in \eqref{eq:BrokVarFormStrong}--\eqref{eq:BrokVarFormPrimal} shows that there is always either an explicit interface variable or sufficient regularity to have well-defined traces of the displacement and stress.


The next task is to more rigorously define the actual \textit{coupled} variational formulations and analyze their well-posedness.
Continuing with the basic example, let $U^{\scP_\mesh}|_{\Omega^\scP}$ be the restriction of the trial space to $\Omega^\scP$ meaning that typical field variables in $U^{\scP_\mesh}_0=U^\scP$ have their domain restricted to $\Omega^\scP$, while the interface variables in $\hat{U}^{\scP_\mesh}$ are restricted to those components associated to elements in $\mesh^\scP$.
Therefore, the space is $U^{\scP_\mesh}|_{\Omega^\scP}=H_{\Gamma_u}^1(\Omega)|_{\Omega^\scP}\times H_{\Gamma_\sigma}^{-\onehalf}(\bdry\mesh)|_{\mesh^\scP}$, with the restricted component norms being ${\|\cdot\|_{H^1(\Omega^\scP)}}$ and ${\|\cdot\|_{H_{\Pi}^{-\onehalf}(\bdry\mesh^\scP)}}$ respectively.
The same applies to $U^{\scU_\mesh}|_{\Omega^\scU}$ and the broken test spaces $V^{\scU_\mesh}|_{\Omega^\scU}$ and $V^{\scP_\mesh}|_{\Omega^\scP}$.
Then, the trial and test spaces associated to the coupled formulations are
\begin{equation}
 	\begin{aligned}
 		U^\scC&=\Big\{\fku^{\scC}=(\fku^\scU,\fku^\scP)\,\Big|\,
 			\fku^\scU=(\sigma^\scU,u^\scU,\omega^\scU,\hat{u}^\scU,\hat{\sigma}_\nml^\scU)\in U^{\scU_\mesh}|_{\Omega^\scU},\,
 				\fku^\scP=(u^\scP,\hat{\sigma}_\nml^\scP)\in U^{\scP_\mesh}|_{\Omega^\scP},\\
 		&\qquad\qquad\qquad\qquad\qquad\qquad\qquad\qquad\qquad\qquad\qquad\quad
 			\hat{u}^\scU|_{\Gamma_I}=u^\scP|_{\Gamma_I},\,
 				\hat{\sigma}_\nml^\scU|_{\Gamma_I}=-\hat{\sigma}_\nml^\scP|_{\Gamma_I}\Big\}\,,\\
 		V^{\scC}&=V^{\scU_\mesh}|_{\Omega^\scU}\times V^{\scP_\mesh}|_{\Omega^\scP}\,.
 	\end{aligned}
 	\label{eq:CoupledTrialAndTestSpaces}
\end{equation}
Hence, the trial space is the subspace of $U^{\scU_\mesh}|_{\Omega^\scU}\times U^{\scP_\mesh}|_{\Omega^\scP}$ which satisfies transmission conditions for both displacement and stress at $\Gamma_I$ (see Remark \ref{rmk:TransmissionConditions} for more details).
On the other hand, the \textit{broken} test space is oblivious to any transmission conditions.
Lastly, the bilinear and linear forms of the coupled variational formulation are
\begin{equation}
 	\begin{aligned}
 		b^{\scC}(\fku^{\scC},\fkv^{\scC})&=b^{\scU_\mesh}|_{\Omega^\scU}(\fku^\scU,\fkv^\scU)+
			b^{\scP_\mesh}|_{\Omega^\scP}(\fku^\scP,\fkv^\scP)\,,\\
 		\ell^{\scC}(\fkv^{\scC})&=\ell^{\scU_\mesh}|_{\Omega^\scU}(\fkv^\scU)+\ell^{\scP_\mesh}|_{\Omega^\scP}(\fkv^\scP)\,,
 	\end{aligned}
 	\label{eq:CoupledBilinearAndLinearForms}
\end{equation}
where the restricted forms $b^{\scU_\mesh}|_{\Omega^\scU}$ and $b^{\scP_\mesh}|_{\Omega^\scP}$ are those formulations in \eqref{eq:BrokVarFormUltraweak} and \eqref{eq:BrokVarFormPrimal} but with the inner products and duality pairings only acting over those elements in $\mesh^\scU$ and $\mesh^\scP$ respectively.
The same applies to the linear forms $\ell^{\scU_\mesh}|_{\Omega^\scU}$ and $\ell^{\scP_\mesh}|_{\Omega^\scP}$.
Evidently, by carefully identifying the trial spaces to enforce the compatibility conditions at the interdomain boundaries, coupled variational formulations can be rigorously generalized to any finite partition of the domain into subdomains, wherein each subdomain is endowed with a broken variational formulation among those found in \eqref{eq:BrokVarFormStrong}--\eqref{eq:BrokVarFormPrimal}.

\begin{remark}
\label{rmk:TransmissionConditions}
There is an abuse of notation when specifying the transmission conditions that enforce the compatibility at the interdomain boundaries in \eqref{eq:CoupledTrialAndTestSpaces}.
More precisely, $\hat{u}^\scU|_{\Gamma_I}=u^\scP|_{\Gamma_I}$ and $\hat{\sigma}_\nml^\scU|_{\Gamma_I}=-\hat{\sigma}_\nml^\scP|_{\Gamma_I}$ denote that there exist global extensions $\tilde{u}\in H^1_{\Gamma_u}(\Omega)$ and $\tilde{\sigma}\in H_{\Gamma_\sigma}(\div,\Omega)$ such that $\tr_{\grad}\tilde{u}|_{\mesh^\scU}=\hat{u}^\scU$, $\tilde{u}|_{\Omega^\scP}=u^\scP$, $\tr_{\div}\tilde{\sigma}|_{\mesh^\scU}=\hat{\sigma}_\nml^\scU$ and $\tr_{\div}\tilde{\sigma}|_{\mesh^\scP}=\hat{\sigma}_\nml^\scP$.
In fact, these global extensions for the displacement and stress, $\tilde{u}$ and $\tilde{\sigma}$, are fundamental in the numerical implementation, where they are considered global variables in the context of a multi-physics domain, whereas the remaining variables only have local support in a particular subdomain.
Moreover, the concept of the extensions is also important for specifying the problem boundary conditions, $u^{\Gamma_u}$ and $\sigma^{\Gamma_\sigma}_\nml$.
Indeed, by definition there exist extensions, $\tilde{u}^{\Gamma_u}\in H^1(\Omega)$ and $\tilde{\sigma}^{\Gamma_\sigma}\in H(\div,\Omega)$, whose appropriate restrictions (e.g. $\tr_{\grad}\tilde{u}^{\Gamma_u}|_{\mesh^\scU}$ and $\tilde{u}^{\Gamma_u}|_{\Omega^\scP}$ for the displacement) play a role in the linear forms $\ell^{\scU_\mesh}|_{\Omega^\scU}$ and $\ell^{\scP_\mesh}|_{\Omega^\scP}$.
\end{remark}

It remains to show that the coupled variational formulations are well-posed.
The technique is similar in spirit to the one utilized in proving well-posedness of broken variational formulations as outlined by \cite[Theorem~3.1]{BrokenForms15}, where the first step is always to test with unbroken test functions to cancel the boundary terms.
The main idea here, however, is to collapse everything to the well-posed ultraweak formulation by testing with more regular test functions and integrating by parts when necessary.
This is interesting since the ultraweak formulation is effectively being used as a tool for a proof, due to its attractive property of having all the weight of the derivatives on the test functions.
Before presenting the proof, a useful and necessary lemma already established in \cite[Appendix~A]{Keith2016} is restated.
%
%

\begin{lemma}
\label{lem:CharacterizationOfTraces}
Let $\mesh$ be an open partition of a domain $\Omega$, and $\,\Gamma_u$ and $\,\Gamma_\sigma$ be relatively open subsets in $\bdry\Omega$ satisfying  $\overline{\Gamma_u\cup\Gamma_\sigma}=\bdry\Omega$ and $\Gamma_u\cap\Gamma_\sigma=\varnothing$.
\begin{enumerate}[font=\upshape,label={(\roman*)},ref={\thelemma(\roman*)}]
	\item Let $v\in H^1(\mesh)$. Then $v\in H^1_{\Gamma_u}(\Omega)$ if and only if $\langle\hat{\sigma}_\nml,\tr_{\grad}v\rangle_{\bdry\mesh}=0$ for all $\hat{\sigma}_\nml\in H^{-\onehalf}_{\Gamma_\sigma}(\bdry\mesh)$. \label{lem:H1Subspace}
	\item Let $\tau\in H(\div,\mesh)$. Then $\tau\in H_{\Gamma_\sigma}(\div,\Omega)$ if and only if $\langle\hat{u},\tr_{\div}\tau\rangle_{\bdry\mesh}=0$ for all $\hat{u}\in H^{\onehalf}_{\Gamma_u}(\bdry\mesh)$. \label{lem:HdivSubspace}
\end{enumerate}
\end{lemma}

\begin{theorem}
	\label{thm:WellPosedness}
Let $\Omega$ be a domain partitioned into a finite number of subdomains, wherein each subdomain is endowed with a broken variational formulation of linear elasticity among those found in \eqref{eq:BrokVarFormStrong}--\eqref{eq:BrokVarFormPrimal}.
Provided $\Gamma_u\neq\varnothing$, the resulting coupled variational formulation is well-posed.
\end{theorem}

\begin{proof}
For the sake of consistency and simplicity the main body of the proof applies to the two-subdomain example described throughout this section, which involves the ultraweak and primal formulations.
Then, a few observations will clarify the more general case.

The goal is to prove that there exists a $\gamma^\scC>0$ such that for every $\fku^\scC=(\fku^\scU,\fku^\scP)\in U^\scC$,
\begin{equation*}
	\gamma^\scC\|\fku^\scC\|_{U^\scC}\leq\sup_{\fkv^\scC\in V^\scC}
		\frac{|b^\scC(\fku^\scC,\fkv^\scC)|}{\|\fkv^\scC\|_{V^\scC}}=\|\fku^\scC\|_E\,,
\end{equation*}
where $\|\fku^\scC\|_E=\|B^\scC\fku^\scC\|_{(V^\scC)'}$ with $B^\scC:U^\scC\to(V^\scC)'$ defined by $\langle B^\scC\fku^\scC,\fkv^\scC\rangle_{(V^\scC)'\times V^\scC} = b^\scC(\fku^\scC,\fkv^\scC)$.
Notice that in the previous expression and in what follows, the zero is omitted tacitly from all suprema.

As usual, the approach is to prove this bound for each of the components in $\fku^\scC=(\fku^\scU,\fku^\scP)$, where ${\fku^\scU=(\fku_0^\scU,\hat{\fku}^\scU)}$, $\fku_0^\scU=(\sigma^\scU,u^\scU,\omega^\scU)\in U^\scU|_{\Omega^\scU}$, $\hat{\fku}^\scU=(\hat{u}^\scU,\hat{\sigma}_\nml^\scU)\in\hat{U}^{\scU_\mesh}|_{\Omega^\scU}$, $\fku^\scP=(\fku_0^\scP,\hat{\fku}^\scP)$, $\fku_0^\scP=u^\scP\in U^\scP|_{\Omega^\scP}$ and $\hat{\fku}^\scP=\hat{\sigma}_\nml^\scP\in\hat{U}^{\scP_\mesh}|_{\Omega^\scP}$.
The first step is to find the bounds for the field variables $\fku_0^\scU$ and $\fku_0^\scP$ by somehow avoiding the terms involving the interface variables.
The main idea to achieve this is to collapse \textit{all} formulations to the ultraweak formulation via careful testing and integration by parts, yielding a \textit{global} ultraweak formulation.
This formulation has all the weight of the derivatives on the test function, so it makes sense to consider the \textit{global} ultraweak test functions $\fkv^\Omega=(\tau,v)\in V^\scU=H_{\Gamma_\sigma}(\div,\Omega)\times H^1_{\Gamma_u}(\Omega)$.

From now on, given any tensor, let the subscripts $\Sym$ and $\Skw$ denote its symmetric and antisymmetric parts.
Let $\omega(u^\scP)=(\nabla u^\scP)_\Skw$, $\varepsilon(u^\scP)=(\nabla u^\scP)_\Sym$, and $\sigma(u^\scP)=\sfC\!:\!\nabla u^\scP$, so that $\varepsilon(u^\scP)=\sfS\!:\!\sigma(u^\scP)$.
Thus, $(\nabla u^\scP,\tau)_{\mesh^\scP}=(\sfS\!:\!\sigma(u^\scP),\tau)_{\mesh^\scP}+(\omega(u^\scP),\tau)_{\mesh^\scP}$, and it follows
\begin{equation*}
	\begin{aligned}
		(\sfC\!:\!\nabla u^\scP,\nabla v)_{\mesh^\scP}
			&\!=\!(\sfS\!:\!\sigma(u^\scP),\tau)_{\mesh^\scP}\!+\!(\omega(u^\scP),\tau)_{\mesh^\scP}\!-\!(\nabla u^\scP,\tau)_{\mesh^\scP}
				\!+\!(\sfC\!:\!\nabla u^\scP,\nabla v)_{\mesh^\scP}\\
			&\!=\!(\sfS\!:\!\sigma(u^\scP),\tau)_{\mesh^\scP}\!+\!(\omega(u^\scP),\tau)_{\mesh^\scP}\!+\!(u^\scP,\div\tau)_{\mesh^\scP}
				\!+\!(\sigma(u^\scP),\nabla v)_{\mesh^\scP}\!-\!\langle\tr_{\grad}u^\scP,\tr_{\div}\tau\rangle_{\bdry\mesh^\scP}\,,
	\end{aligned}
\end{equation*}
where integration by parts was valid due to the high regularity of both $u^\scP$ and $\tau$.
Therefore
\begin{equation*}
	b^{\scP_\mesh}|_{\Omega^\scP}(\fku^\scP,\fkv^\scP)=b^\scU|_{\Omega^\scP}(\fku^{\scU_\scP},\fkv^\Omega)
		-\langle\tr_{\grad}u^\scP,\tr_{\div}\tau\rangle_{\bdry\mesh^\scP}
			-\langle\hat{\sigma}_\nml^\scP,\tr_{\grad}v\rangle_{\bdry\mesh^\scP}\,,
\end{equation*}
where $\fku^\scP=(u^\scP,\hat{\sigma}_\nml^\scP)$, $\fku^{\scU_\scP}=(\sigma(u^\scP),u^\scP,\omega(u^\scP))$, $\fkv^\Omega=(\tau,v)$ and $\fkv^\scP=v$.
Trivially it holds that
\begin{equation*}
	b^{\scU_\mesh}|_{\Omega^\scU}(\fku^\scU,\fkv^\scU)=b^\scU|_{\Omega^\scU}(\fku^{\scU_\scU},\fkv^\Omega)
		-\langle\hat{u}^\scU,\tr_{\div}\tau\rangle_{\bdry\mesh^\scU}
			-\langle\hat{\sigma}_\nml^\scU,\tr_{\grad}v\rangle_{\bdry\mesh^\scU}\,,
\end{equation*}
where $\fku^\scU=(\fku_0^\scU,(\hat{u}^\scU,\hat{\sigma}_\nml^\scU))$, $\fku^{\scU_\scU}=\fku_0^\scU=(\sigma^\scU,u^\scU,\omega^\scU)$, $\fkv^\Omega=\fkv^\scU=(\tau,v)$.
Adding the previous two expressions and using the transmission conditions for the displacement and stress (see Remark~\ref{rmk:TransmissionConditions}) along with Lemma~\ref{lem:CharacterizationOfTraces}, it follows that the interface terms vanish, resulting in
\begin{equation*}
	b^\scC(\fku^\scC,\fkv^{\scC_\Omega})=b^\scU(\fku^\Omega,\fkv^\Omega)\,,
\end{equation*}
where $\fku^\scC=((\sigma^\scU,u^\scU,\omega^\scU,\hat{u}^\scU,\hat{\sigma}_\nml^\scU),(u^\scP,\hat{\sigma}_\nml^\scP))\in U^\scC$, $\fku^\Omega=(\sigma^\Omega,u^\Omega,\omega^\Omega)$ is defined by $\fku^\Omega|_{\Omega^\scU}=(\sigma^\scU,u^\scU,\omega^\scU)$ and $\fku^\Omega|_{\Omega^\scP}=(\sigma(u^\scP),u^\scP,\omega(u^\scP))$, $\fkv^\Omega=(\tau,v)$ and $\fkv^{\scC_\Omega}=((\tau,v)|_{\Omega^\scU},v|_{\Omega^\scP})$.
Thus, when testing appropriately, the coupled formulation is essentially a global ultraweak formulation.

The global ultraweak variational formulation is known to be well-posed when $\Gamma_u\neq\varnothing$ (see \cite[Theorem~2.1]{Keith2016}), so that there exists $\gamma^\scU$ such that
\begin{equation*}
	\gamma^\scU\|\fku^\Omega\|_{U^\scU}
		\leq\sup_{\fkv^\Omega\in V^\scU}\frac{|b^\scU(\fku^\Omega,\fkv^\Omega)|}{\|\fkv^\Omega\|_{V^\scU}}
	  	=\sup_{\fkv^\Omega\in V^\scU}\frac{|b^\scC(\fku^\scC,\fkv^{\scC_\Omega})|}{\|\fkv^\Omega\|_{V^\scU}}
	  		\leq\sup_{\fkv^\Omega\in V^\scU}\frac{|b^\scC(\fku^\scC,\fkv^{\scC_\Omega})|}{\|\fkv^{\scC_\Omega}\|_{V^\scC}}
	  			\leq \|\fku^\scC\|_E\,,
\end{equation*}
where it is used that $\|\fkv^{\scC_\Omega}\|_{V^\scC}\leq \|\fkv^\Omega\|_{V^\scU}$.
Naturally, $\|(\sigma^\scU,u^\scU,\omega^\scU)\|_{U^\scU|_{\Omega^\scU}}\leq\|\fku^\Omega\|_{U^\scU}$.
Meanwhile, since $\|\varepsilon(u^\scP)\|_{L^2(\Omega^\scP;\Sym)}\leq\|\sfS\|\|\sigma(u^\scP)\|_{L^2(\Omega^\scP;\Sym)}$, 
it follows that $\|u^\scP\|_{H^1(\Omega^\scP)}\!=\!\|(\varepsilon(u^\scP),u^\scP,\omega(u^\scP))\|_{U^\scU|_{\Omega^\scP}}\leq C_\sfS\|\fku^\Omega\|_{U^\scU}$, where $C_\sfS=\max\{1,\|\sfS\|\}$.
Thus, there exist constants $C^\scU>0$ and $C^\scP>0$, such that
\begin{equation*}
	\|\fku_0^\scU\|_{U^\scU|_{\Omega^\scU}}\leq C^\scU \|\fku^\scC\|_E\,,\qquad\qquad
		\|\fku_0^\scP\|_{U^\scP|_{\Omega^\scP}}\leq C^\scP \|\fku^\scC\|_E\,,
\end{equation*}
for all $\fku_0^\scU=(\sigma^\scU,u^\scU,\omega^\scU)$ and $\fku_0^\scP=u^\scP$.

The last step involves finding the bounds for the interface variables, $\hat{\fku}^\scU=(\hat{u}^\scU,\hat{\sigma}_\nml^\scU)$ and $\hat{\fku}^\scP=\hat{\sigma}_\nml^\scP$.
Consider $\hat{\sigma}_\nml^\scP$ and let $M_0^\scP$ be a continuity bound satisfying $\big|b_0^{\scP_\mesh}|_{\Omega^\scP}(\fku_0^\scP,\fkv^\scP)\big|\leq M_0^\scP \|\fku_0^\scP\|_{U^\scP|_{\Omega^\scP}}\|\fkv^\scP\|_{V^{\scP_\mesh}|_{\Omega^\scP}}$ for all $\fku_0^\scP\in U^\scP|_{\Omega^\scP}$ and $\fkv^\scP\in V^{\scP_\mesh}|_{\Omega^\scP}$.
Then, using the identities in \eqref{eq:traceidentitiesmesh}, it follows
\begin{equation*}
	\begin{aligned}
		\|\hat{\sigma}_\nml^\scP\|_{H^{-\onehalf}_\Pi(\bdry\mesh^\scP)}
			&=\sup_{v\in H^1(\mesh^\scP)}\frac{|\langle\hat{\sigma}_\nml^\scP,\tr_{\grad}v\rangle_{\bdry\mesh^\scP}|}{\|v\|_{H^1(\mesh^\scP)}}
				\leq\sup_{\fkv^\scP\in V^{\scP_\mesh}|_{\Omega^\scP}}
					\frac{\big|b^{\scP_\mesh}|_{\Omega^\scP}(\fku^\scP,\fkv^\scP)-b_0^{\scP_\mesh}|_{\Omega^\scP}(\fku_0^\scP,\fkv^\scP)\big|}
						{\|\fkv^\scP\|_{V^{\scP_\mesh}|_{\Omega^\scP}}}\\
		&\leq \|\fku^\scC\|_E+M_0^\scP\|\fku_0^\scP\|_{U^\scP|_{\Omega^\scP}}
			\leq \big(1+M_0^\scP C^\scP\big)\|\fku^\scC\|_E\,.
	\end{aligned}				
\end{equation*}
Similar calculations hold for $\hat{u}^\scU$ and $\hat{\sigma}_\nml^\scU$.
Summing the contributions from $\fku_0^\scU$, $\hat{\fku}^\scU$, $\fku_0^\scP$ and $\hat{\fku}^\scP$, yields the desired constant $C^\scC=\frac{1}{\gamma^\scC}>0$ satisfying $\|\fku^\scC\|_{U^\scC}\leq C^\scC \|\fku^\scC\|_E$ for every $\fku^\scC=((\fku_0^\scU,\hat{\fku}^\scU),(\fku_0^\scP,\hat{\fku}^\scP))\in U^\scC$.

The more general case follows analogously, but some technicalities, mostly arising form the weak imposition of symmetry, are worth mentioning.
To observe the changes, it suffices to consider the two-subdomain case involving the strong and ultraweak formulations.
In this case, a similar procedure as before yields
\begin{equation*}
	b^\scC(\fku^\scC,\fkv^{\scC_\Omega})=b^\scU(\fku^\Omega,\fkv^\Omega)+(\sigma^\scS_\Skw,\nabla v)_{\Omega^\scS}\,,
\end{equation*}
where $\fku^\scC=((\sigma^\scS,u^\scS),(\sigma^\scU,u^\scU,\omega^\scU,\hat{u}^\scU,\hat{\sigma}_\nml^\scU))\in U^\scC$, $\fku^\Omega=(\sigma^\Omega,u^\Omega,\omega^\Omega)$ is defined by $\fku^\Omega|_{\Omega^\scS}=(\sigma^\scS_\Sym,u^\scS,\omega(u^\scS))$ and $\fku^\Omega|_{\Omega^\scU}=(\sigma^\scU,u^\scU,\omega^\scU)$, $\fkv^\Omega=(\tau,v)$ and $\fkv^{\scC_\Omega}=((\sfS\!:\!\tau,v,0)|_{\Omega^\scS},(\tau,v)|_{\Omega^\scU})$.
Thus, in order to obtain the bound of $\|\fku_0^\scS\|_{U^\scS|_{\Omega^\scS}}=\|(\sigma^\scS,u^\scS)\|_{H(\div,\Omega^\scS)\times H^1(\Omega^\scS)}$ in terms of $\|\fku^\scC\|_E$, it is enough to bound $\|\sigma^\scS_\Skw\|_{L^2(\Omega^\scS;\Skw)}$, $\|\div\sigma^\scS\|_{L^2(\Omega^\scS)}$ and $\|\varepsilon(u^\scS)\|_{L^2(\Omega^\scS;\Sym)}$ in terms of $\|\fku^\scC\|_E$.
These bounds are easily obtained through a careful choice of test functions in $\Omega^\scS$.
With these facts, the astute reader can deduce the proof for any other relevant and general scenario.
\end{proof}

\begin{remark}
As mentioned before, in this work, the variational formulations of linear elasticity avoid the strong imposition of tensor symmetry in some of the spaces.
This, among other reasons, adds a layer of complexity to the formulations and the corresponding proofs, which typically need a few extra calculations.
However, these difficulties are not present in many other important equations.
Indeed, a simpler version of this proof can easily be applied to coupled formulations of Poisson's equation, time-harmonic Maxwell's equations (see \cite{BrokenForms15} for multiple formulations) and the diffusion-convection-reaction equation (see \cite{DemkowiczClosedRange} for multiple formulations), among others.
\end{remark}

\begin{remark}
{
The stability constant in the proof of the theorem, $\gamma^\scC$, depends on the distribution of the variational formulations and the shape of the subdomains.
The constant will remain robust with respect to heterogeneous material properties as long as each formulation is robust when viewed independently.
Thus, the stability constant will remain bounded above as long as any near or fully incompressible elastic behavior is limited to subdomains associated to robustly well-posed variational formulations (i.e.~broken ultraweak and mixed formulations).} 
\end{remark}

\section{The DPG methodology}
\label{sec:DPG}

The DPG methodology is a minimum residual finite element method.
Minimum residual methods are posed in the most general sense by considering an abstract linear variational formulation as in \eqref{eq:bilinearFormEQ}, and noting that it can be restated as an operator equation,
\begin{equation}
\label{eq:linearFormEQ}
 \left\{
  \begin{aligned}
   &\text{Find } \fku\in U\,,\\
   &B\fku = \ell\,,
  \end{aligned}
 \right.
\end{equation}
where $B:U\to V'$ is a linear continuous operator defined by $\langle B\fku,\fkv\rangle_{V'\times V} = b(\fku,\fkv)$.
Then, once a discrete trial space $U_h\subseteq U$ is chosen, the method simply seeks the minimizer of the residual,
\begin{equation}
	\fku_h = \argmin_{\fku\in U_h} \|B\fku-\ell\|^2_{V'}\,.
	\label{eq:MinResidual}
\end{equation}
This minimization problem can be recast in several equivalent ways. 
Amongst them are those that seek $\fku_h\in U_h$ such that
\begin{subequations}
	\begin{align}
		\big(B\fku_h,B\delta\fku\big)_{V'}&=\big(\ell,B\delta\fku\big)_{V'} \,,\label{eq:DPGMinVirgin}\\
		\big\langle B\fku_h,R_V^{-1}B\delta\fku\big\rangle_{V'\times V}&=
			\big\langle\ell,R_V^{-1}B\delta\fku\big\rangle_{V'\times V}\,,\label{eq:DPGMinOptTest}\\
		\big\langle\delta\fku,B^\dagger R_V^{-1}B\fku_h\big\rangle_{U\times U'}&=
			\big\langle\delta\fku,B^\dagger R_V^{-1}\ell\big\rangle_{U\times U'}\,,\label{eq:DPGMinImplement}
	\end{align}
\end{subequations}
for all $\delta\fku\in U_h$.
Here $R_V:V\to V'$ is the Riesz map of $V$, which is defined by $\langle R_V\fkv,\delta\fkv\rangle_{V'\times V}=(\fkv,\delta\fkv)_V$ for all $\fkv,\delta\fkv\in V$, and which is known to be an isometric isomorphism.
Meanwhile, $B^\dagger:V\to U'$ is the (conjugate) transpose of $B$, which is defined by $\langle \fku,B^\dagger \fkv\rangle_{U\times U'}=b(\fku,\fkv)=\langle B\fku,\fkv\rangle_{V'\times V}$.
With the same notation, the minimum residual value over $U_h$ can be expressed as,
\begin{equation}
  \|B\fku_h-\ell\|_{V'}^2=\big\langle B\fku_h-\ell,R_V^{-1}(B\fku_h-\ell)\big\rangle_{V'\times V}
  	=\big\langle\fku_h-\fku,B^\dagger R_V^{-1}(B\fku_h-\ell)\big\rangle_{U\times U'}\,,
  \label{eq:ExactResidual}
\end{equation}
where $\fku\in U$ is the exact solution satisfying $B\fku=\ell$.
Letting $\delta\fkv=R_V^{-1}B\delta\fku\in R_V^{-1}BU_h=V^\opt$, one can reinterpret \eqref{eq:DPGMinOptTest} as a fully discrete variational formulation with bilinear and linear forms $b|_{U_h\times V^\opt}$ and $\ell|_{V^\opt}$.
This formulation has the paramount property that, despite being posed over finite-dimensional spaces ($\dim(U_h)=\dim(V^\opt)<\infty$) it shares the same stability properties as the original variational formulation, which is the best one can hope for \cite{DPGOverview}.
These are the reasons why the space $V^\opt$ is called the optimal test space, and why these minimum residual methods are sometimes referred to as \textit{the} optimal Petrov-Galerkin methods.

Unfortunately, any attempt at solving the variational systems inevitably encounters the exact inversion of the Riesz map of the test space, $R_V^{-1}$, which is typically impossible to manage numerically due to the infinite-dimensional nature of $V$.
Thus, the idea is to consider an inversion of the Riesz map over a large, yet still finite-dimensional subspace of $V$, called the enriched test space, $V^\enr\subseteq V$, which satisfies that $\dim(V^\enr)\geq\dim(U_h)$.
In this case, the approximate optimal test space $V^\opt_h=R_{V^\enr}^{-1}BU_h$ aims to be as close as possible to the actual optimal test space, $V^\opt$.
The larger $\dim(V^\enr)$ is, the closer $V^\opt_h$ is from $V^\opt$. 
In practice, it is more convenient to solve the discrete operator equation in \eqref{eq:DPGMinImplement},
\begin{equation}
	B^\dagger R_{V^\enr}^{-1}B\fku_h=B^\dagger R_{V^\enr}^{-1}\ell\,,
	\label{eq:DiscreteNormalEq}
\end{equation}
where numerically speaking, $(B)_{ij}=b(\fku_j,\fkv_i)$, $(\ell)_i=\ell(\fkv_i)$, and $(R_{V^\enr})_{ij}=(\fkv_i,\fkv_j)_V$, with $(\fku_j)_{j=1}^{\dim(U_h)}$ and $(\fkv_i)_{i=1}^{\dim(V^\enr)}$ being bases for $U_h$ and $V^\enr$ respectively, and $(\fku_h)_j$ being the vector of coefficients of the solution $\fku_h$ with respect to the $(\fku_j)_{j=1}^{\dim(U_h)}$ basis.
Notice it is not necessary to find a basis for the approximate optimal test space, $V^\opt_h$, since this is implicitly calculated within \eqref{eq:DiscreteNormalEq} from a basis of the enriched test space $V^\enr$, $(\fkv_i)_{i=1}^{\dim(V^\enr)}$.
Typically, Petrov-Galerkin methods require that one produce an explicit basis for the ``exotic'' discrete test space being proposed (in this case being $V^\opt_h$), but this method is distinguished from those in that it essentially \textit{calculates} such a basis automatically from a ``standard'' (yet large) basis of a more commonplace discretization of the test space (i.e.~that of $V^\enr$).
Moreover, notice the stiffness matrix associated to the problem, $B^\dagger R_{V^\enr}^{-1}B=(R_{V^\enr}^{-\onehalf}B)^\dagger (R_{V^\enr}^{-\onehalf}B)$, is a symmetric (or Hermitian) positive definite matrix, which is extremely convenient.
In fact, the unique Cholesky decomposition, $R_{V^\enr}=LL^\dagger$, where $L$ is a lower triangular matrix with positive diagonal entries, implies that $L$ acts similar to a square root of $R_{V^\enr}$, with $R_{V^\enr}^{-1}=L^{-\dagger}L^{-1}$ and
\begin{equation}
	B^\dagger R_{V^\enr}^{-1}B=(L^{-1}B)^\dagger(L^{-1}B)\,.
	\label{eq:CholeskyStiffness}
\end{equation}
The key here is that $L^{-1}B$ is computed using forward substitution, and then it is simply premultiplied with its own (conjugate) transpose to produce the stiffness matrix in \eqref{eq:CholeskyStiffness}.
Computationally, this can be made more efficient than other alternatives, such as computing $R_{V^\enr}^{-1}B=L^{-\dagger}L^{-1}B$ and then premultiplying by $B^\dagger$.
Lastly, the discrete minimum residual can be computed by use of the expression
\begin{equation}
	\|B\fku_h-\ell\|_{(V^\enr)'}^2=(B\fku_h-\ell)^\dagger R_{V^\enr}^{-1}(B\fku_h-\ell)
		=(L^{-1}(B\fku_h-\ell))^\dagger(L^{-1}(B\fku_h-\ell))\,.
	\label{eq:DiscreteResidual}
\end{equation}

Regrettably, the inversion of the Riesz map of $V^\enr$ still represents a global problem over the whole domain, and this increases the numerical cost of the method greatly, to the point that it appears impractical.
In this sense, the use of broken test spaces comes to the rescue.
Indeed, by tackling broken variational formulations that make use of broken test spaces, the computation of the inverse of the Riesz map is decoupled and can be performed locally at each element.
In fact, the equations in \eqref{eq:DiscreteNormalEq} can be solved locally and then assembled into a global stiffness matrix, with the local stiffness matrix being constructed efficiently by making use of the decomposition in \eqref{eq:CholeskyStiffness}.
The same argument applies to the residual in \eqref{eq:DiscreteResidual}, which can now be computed at each element, and can be used to drive an adaptivity scheme since globally it is expected to decrease as $U_h$ is refined.
The cost of using broken test spaces is that broken variational formulations have new interface variables within the trial space which adds to the number of degrees of freedom, but, fortunately, they only act on the boundary of the elements of the mesh.
The optimal Petrov-Galerkin methods coming from this minimum residual approach, coupled with the use of broken test spaces, is referred to as the DPG methodology.

\begin{remark}
\label{rmk:ExactSequence}
Usually variational formulations make use of the typical Hilbert spaces described in Section~\ref{sec:energyspaces} along with the $H(\curl,K)$ spaces and their variants.
The basic underlying spaces are related to each other through an exact sequence, which in three dimensions is
\begin{equation*}
	H^1(K) \xrightarrow{\,\,\nabla\,\,} H(\curl,K) \xrightarrow{\nabla\times} H(\div,K) \xrightarrow{\,\nabla\cdot\,} L^2(K) \, .
\end{equation*}
Fortunately, for the typical element shapes (i.e.~tetrahedra, hexahedra, triangular prisms and pyramids), it is possible to find discrete analogues of these spaces that satisfy the same exact sequence property and a family of interpolation operators that commute with the sequence above \cite{hpbook2,Fuentes2015,Cockburn16}.
Moreover, the discrete sequence of spaces can be made to have an associated polynomial order, $p$, so that it more closely approximates the infinite-dimensional sequence as $p$ grows.
Lastly, the traces of the discrete sequence of spaces form a sequence themselves and serve as discretizations to the trace Hilbert spaces defined in Section~\ref{sec:energyspaces} \cite{Fuentes2015}.
For those reasons, in practice, the discretizations of the trial and test spaces are drawn from these discrete sequences of spaces.
In fact, the trial spaces, $U_h$, are derived from a sequence of order $p$, while the larger test spaces, $V^\enr$, are taken from an enriched sequence of order $p+\mathrm{d}p$, where $\mathrm{d}p$ represents an enrichment parameter.
For the trial spaces, one must be careful when choosing shape functions that form bases for the discrete spaces of a given element, since they must be compatible at the faces with shape functions coming from neighboring elements (possibly of different shapes).
Fortunately, some of those carefully constructed bases can be found in the literature, and for the computations used in this work, the shape functions spanning the exact sequence spaces were chosen from \cite{Fuentes2015}.
Regarding $\mathrm{d}p$, this parameter should be large enough such that the approximate test space, $V^\opt_h$, is close enough to the actual optimal test space, $V^\opt$, to ensure that the underlying numerical method is stable. 
Theoretically, this is always possible, but might involve a very large parameter $\mathrm{d}p$, making the local computations of the method considerably less efficient.
Thankfully, a parameter of $1\leq\mathrm{d}p\leq3$ is usually sufficient in practice \cite{gopalakrishnan2014analysis,BrokenForms15}, and in fact in this work a value of $\mathrm{d}p=1$ was sufficient to ensure numerical stability.
Finally, for all trial variables derived from a sequence of order $p$ (including interface variables), standard norm interpolation estimates of the form $Ch^p$ are guaranteed, where $C$ is a constant and $h$ represents the mesh size.
This implies convergence rates of the same order given numerical stability.
\end{remark} 

\begin{remark}
\label{rmk:L2Optimization}
{When \textit{part} of the test space is some version of $L^2$ with its standard norm, the inverse of that component of the Riesz map is actually known exactly and its computation can be avoided altogether.}
This allows one to obtain a more precise approximate optimal test space.
The details on improving the implementation of these scenarios can be found in \cite[Section 4.3, Remark 4.6]{Keith2016}.
Indeed, when the \textit{full} test space is a version of $L^2$ one can manage to exactly reproduce the optimal test space, and this is generally referred to in the literature as a first order system least squares (FOSLS) formulation.
\end{remark}

\begin{remark}
It must be noted that, in theory, the solution converges optimally in the problem-dependent energy norm, $\|\cdot\|_{E}=\|B(\cdot)\|_{V'}$, since $\|\fku_h-\fku\|_{E}=\|B\fku_h-\ell\|_{V'}$ is precisely the residual that is being minimized.
Thus, the behavior of the convergence is in some sense dependent on the norm of the test space $\|\cdot\|_V$.
In practice, the choice of the test norm can have profound repercussions in the effectiveness with which $V^\opt_h$ approximates $V^\opt$, as well as in the number of refinements required to achieve a desired error bound of a particular variable.
However, in this work, for the purpose of simplicity, the norms of the broken test spaces simply come from the standard norms in \eqref{eq:BrokenSpacesAndNorms}.
It would be interesting to investigate the advantages of using more unusual Hilbert norms for the same test spaces.
\end{remark}
\section{Results and discussion} 
\label{sec:results}

In this section, two illustrative examples involving coupled variational formulations as introduced in Section~\ref{sec:coupled} were solved using the DPG methodology detailed in Section~\ref{sec:DPG}.
First, a smooth manufactured solution on a cube with uniform and contrived material data involving four distinct variational formulations was considered.
Then, a more physically motivated and challenging example was tackled: a sheathed hose with large material and layer-thickness contrast, and with one layer composed of a fully incompressible material.
In both of these scenarios, exact solutions were derived, so that convergence was accurately displayed based on the chosen measure of error.

It should be emphasized to the reader that despite all variables and equations having been nondimensionalized, the complications of multiple scales in the formulations and associated norms are not alleviated.
This issue may drastically affect practical computations, but since the examples are given only for an illustration of feasibility, a detailed study of the scaling was overlooked in this work.
Indeed, the formulations used as building blocks for the coupled formulations were taken directly from \eqref{eq:BrokVarFormStrong}--\eqref{eq:BrokVarFormPrimal}, and the norms used were the standard norms in \eqref{eq:BrokenSpacesAndNorms}.
The convergence analysis presented here is limited to uniform mesh refinements.
In fact, the traditional DPG-style residual-derived adaptive mesh refinements (see \eqref{eq:DiscreteResidual}) are out of place here, since the use of the standard norms introduces a subdomain bias in the residual estimates.
For this reason, in the future, a more careful analysis involving the choice of norms is required to anticipate refinement patterns which adequately adapt to solution features.


All computations were performed with the finite element software \textit{hp3d} which has support for local $h$ and $p$ refinements \cite{hpbook,hpbook2}, high order exact sequence shape functions for elements of various shapes \cite{Fuentes2015}, sophisticated multi-physics support (facilitating the global assembly necessary to make the trial space identification for coupled formulations as suggested by Remark~\ref{rmk:TransmissionConditions}), projection-based interpolation of the entire exact sequence \cite{demkowicz2008polynomial} (to enforce non-homogeneous boundary conditions), and the ability to handle isoparametric geometries with curvilinear inherited interior refinements \`{a} la local transfinite interpolation. 
This last feature, allowed to produce uniform refinements of the very thin outer layer of the sheathed hose described in the second example.
In both examples only hexahedral elements were utilized, with the discrete trial and test spaces taken as described in Remark~\ref{rmk:ExactSequence}, where a value of $\mathrm{d}p=1$ was used for all computations.
Locally, the discretizations were identical to those discussed in \cite[Section~5]{Keith2016}, and in the global assembly of the coupled formulations, continuity between neighboring subdomains of variables in $H^1$, $H(\div)$, $H^{1/2}$ and $H^{-1/2}$ was enforced as discussed in Remark~\ref{rmk:TransmissionConditions}.

%

\subsection{Cube domain} 
\label{sec:manufactured_solution}

\begin{figure}[!ht]
 	\centering
  \includegraphics[trim=0cm 0cm 0cm 0cm,clip=true,scale=0.39]{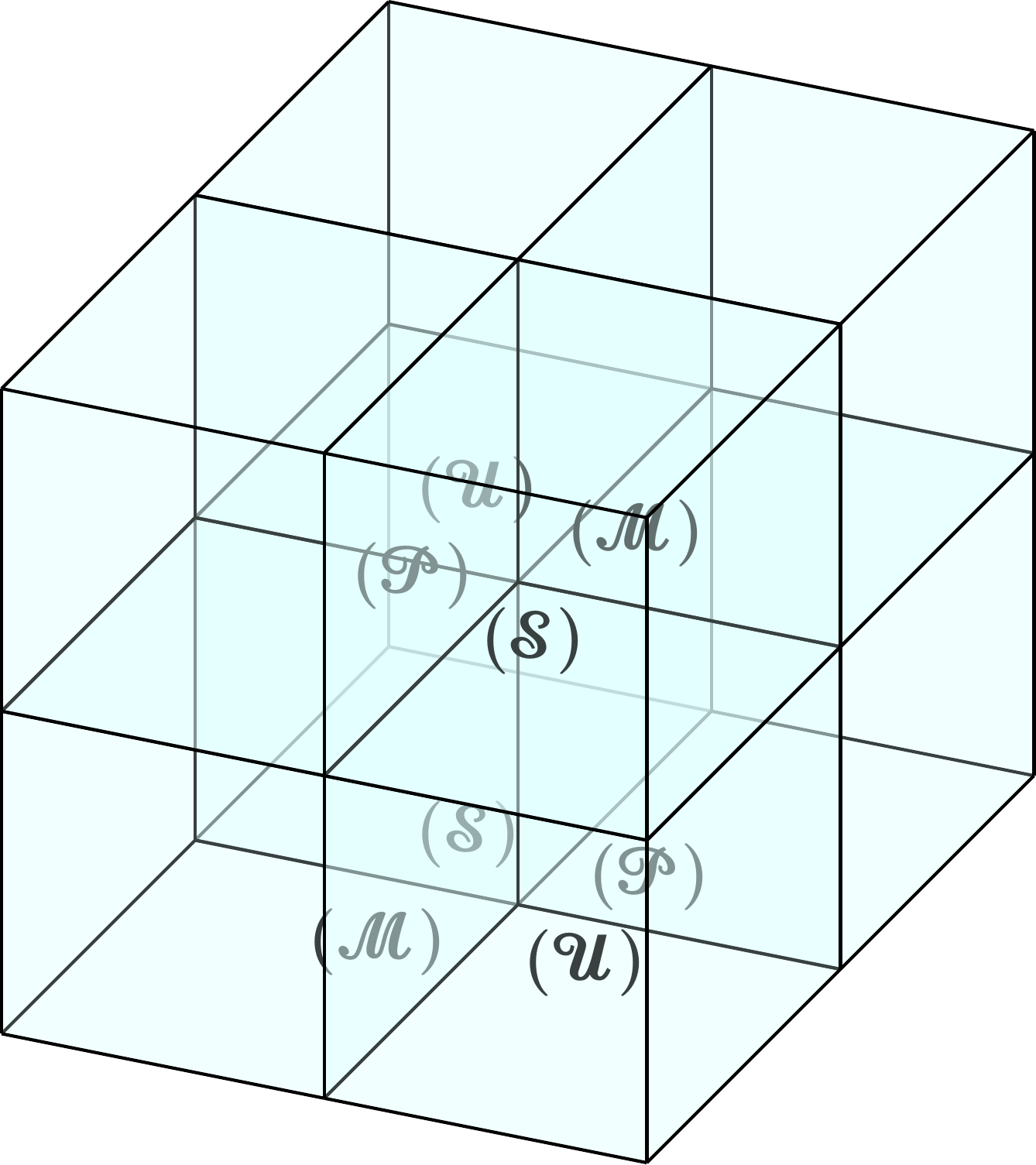}
  \caption{
	Illustration of the geometry and arrangement of subdomains used for a coupled formulation with which the code was verified via manufactured solutions. 
	}
  \label{fig:cubic-domain}
\end{figure}

\noindent As alluded previously, a smooth manufactured solution was taken for the displacement. 
It was a simple sinusoidal vector field,
\begin{equation}
    u_i(x_1,x_2,x_3)=\sin(\pi x_1)\sin(\pi x_2)\sin(\pi x_3)\,, \qquad i=1,2,3,
    \label{eq:exactdisplacementcube}
\end{equation}
on the cubic domain $\Omega=(0,2)^3$. 
The material was considered to be isotropic and homogeneous with nondimensionalized Lam\'e parameters $\lambda=\mu=1$.
Then, partitioning the domain into eight equally sized unit cube subdomains, a configuration was formed for which four distinct broken formulations interact with each other.
More specifically, the strong ($\scS$), ultraweak ($\scU$), mixed ($\scM$) and primal ($\scP$) broken formulations (see \eqref{eq:BrokVarFormStrong}--\eqref{eq:BrokVarFormPrimal}) were organized such that there is at least one face that is a common interface between each of the possible pairs of formulations, as shown in Figure~\ref{fig:cubic-domain}.
The strong and mixed formulations were solved more efficiently as indicated in Remark~\ref{rmk:L2Optimization}.
Finally, the displacement boundary data was prescribed along the whole boundary of $\Omega$.


To analyze convergence, only the $L^2(\Omega)$ error of the displacement was considered and as Figure~\ref{fig:cube-uniform-displacement-error} demonstrates, $p$-th order (or better) convergence rates were witnessed for $p$-order schemes for $1\leq p\leq 5$. 
This is consistent with the theoretical expectations as dictated by Remark~\ref{rmk:ExactSequence}.

\begin{figure}[!ht]
 	\centering
  \includegraphics[trim=0cm 0cm 0cm 0cm,clip=true,scale=0.39]{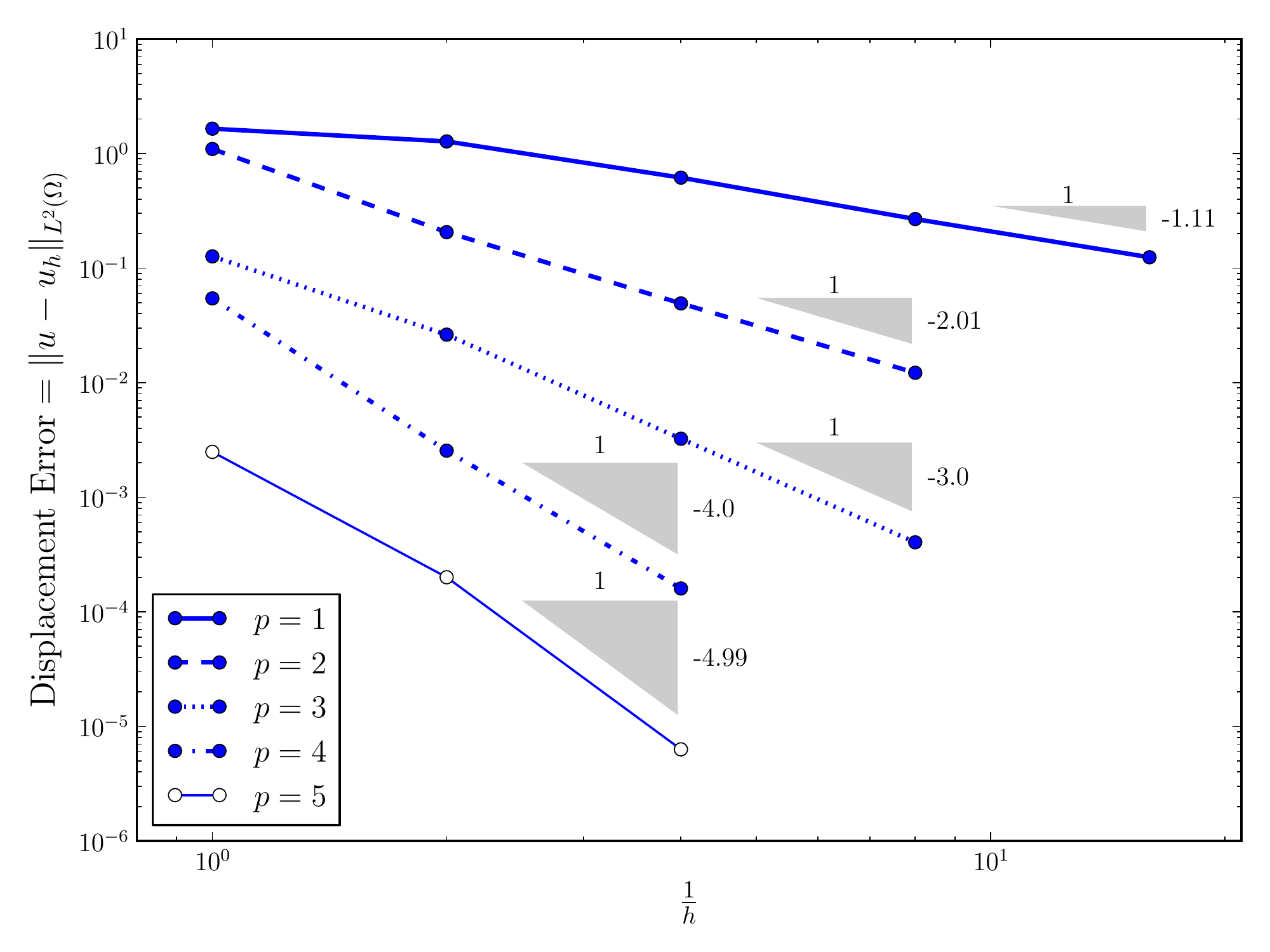}
  \caption{Displacement error as a function of the mesh size under uniform hexahedral refinements in the cubic domain $\Omega=(0,2)^3$ with a sinusoidal manufactured solution.}
  \label{fig:cube-uniform-displacement-error}
\end{figure}


\subsection{Sheathed hose} 
\label{sec:SheathedHose}

\begin{figure}[!ht]
 	\centering
  \includegraphics[trim=0cm 0cm 0cm 0cm,clip=true,scale=0.67]{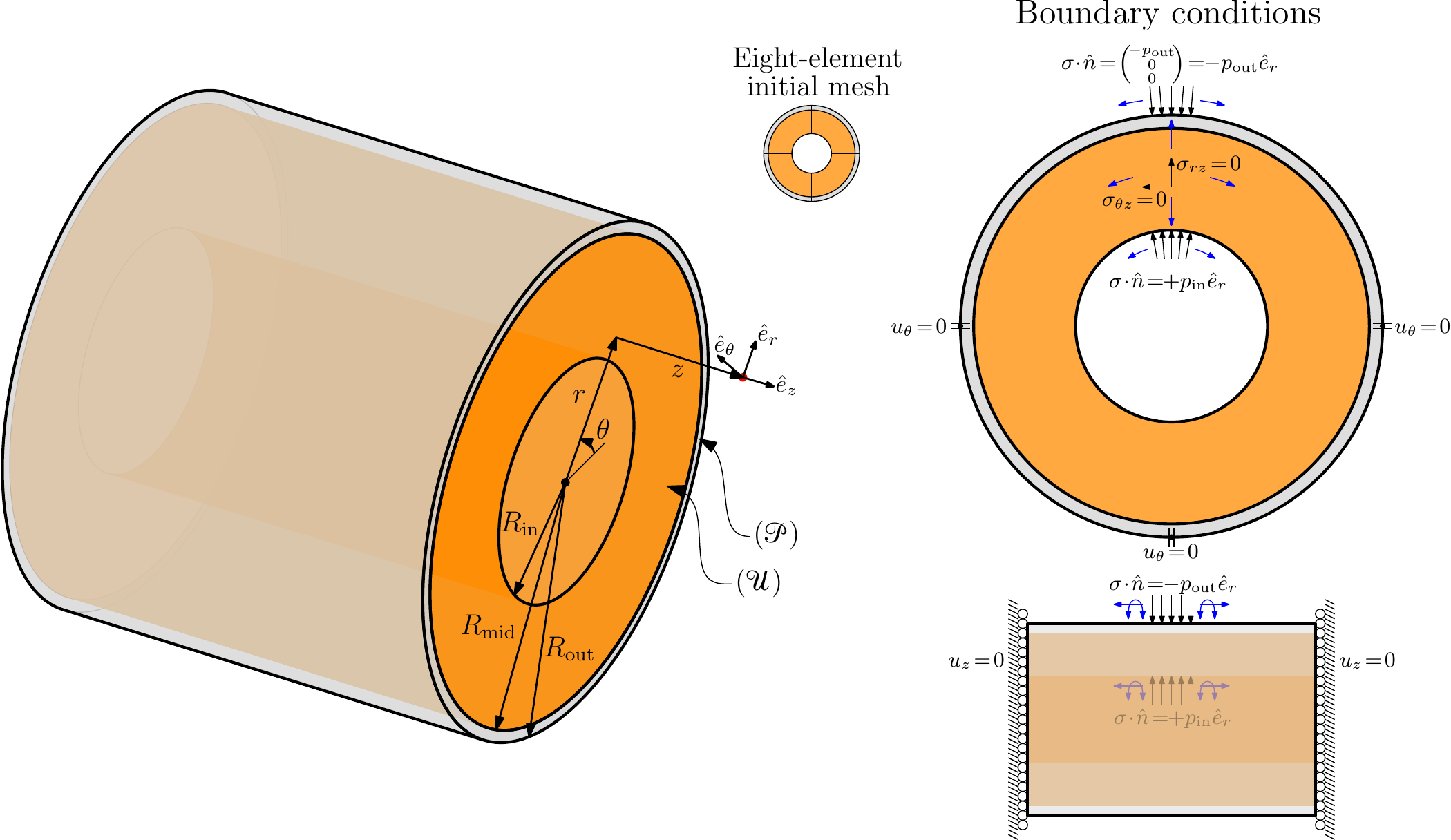}
  \caption{Diagram of the sheathed hose problem with the configuration of variational formulations per subdomain and a schematic of the boundary conditions used.}
  \label{fig:SheathedHose}
\end{figure}

\noindent As a second example, a rubber hose (hollow cylinder) sheathed by a layer of steel was considered as illustrated in Figure~\ref{fig:SheathedHose}.
{
This is a more physically-relevant example, because similar configurations are used in an array of applications including high-performance racing engine hoses and high-pressure hydraulic oil hoses (see SAE hydraulic hose standards), which have a steel braided outer sleeve.
A variation of the example may also be pertinent to stents inside an artery.
}

To simulate balanced axial stresses which would appear in an infinite tube, the axial faces were confined by vanishing normal direction (axial) displacement boundary conditions and zero traction boundary conditions in the tangential directions as depicted in Figure~\ref{fig:SheathedHose}.
Moreover, normal pressure distributions were placed on the inner rubber surface, $p_\mathrm{in}(\theta,z)$ at $R_\mathrm{in}$, and on the outer steel surface, $p_\mathrm{out}(\theta,z)$ at $R_\mathrm{out}$, where $\theta$ represents the azimuthal direction and $z$ represents the axial direction.
Lastly, to have a supported structure with a fixed origin, three non-collinear points of one of the axial faces had their azimuthal displacement set to zero.

The thickness of the outer steel layer, $R_\mathrm{out}-R_\mathrm{mid}$, was assumed to be much smaller than the thickness of the rubber layer, $R_\mathrm{mid}-R_\mathrm{in}$.
This implied the use of very thin elements provided each layer was discretized by the same number of elements in the radial direction, so that shear locking could have been a concern.
Additionally, the rubber was taken to be the demanding case of a fully incompressible material, so that volumetric locking also had to be avoided.
In principle, this makes the problem particularly challenging to solve, and therefore constitutes an ideal testing ground for the method being analyzed.
Fortunately, the use of coupled variational formulations was extremely convenient, because, at least with regard to volumetric locking, all one needed to do was to choose a robustly well-posed variational formulation for the rubber subdomain coupled with a more efficient formulation in the steel subdomain, where robustness with respect to the material properties was not an issue.
In fact, as shown in Figure~\ref{fig:SheathedHose}, the broken ultraweak formulation ($\scU$) was chosen for the rubber, while the broken primal formulation ($\scP$) was chosen for the steel.  

For reference, the steel had a Young's modulus of $E_\mathrm{S} = 200\,\mathrm{GPa}$ and a Poisson's ratio of $\nu_\mathrm{S} = 0.285$, and the rubber had a Young's modulus of $E_\mathrm{R} = 0.01\,\mathrm{GPa}$ and a Poisson's ratio of $\nu_\mathrm{R} = 0.5$.
Then, the Lam\'e parameters were easily {calculated} using the formulas $\lambda=\frac{E\nu}{(1+\nu)(1-2\nu)}$ and $\mu=\frac{E}{2(1+\nu)}$.
Meanwhile, the radii used were $R_\mathrm{in}=0.5\,\mathrm{m}$, $R_\mathrm{mid}=0.99\,\mathrm{m}$ and {$R_\mathrm{out}=1.0\,\mathrm{m}$}.

\subsubsection{Uniform pressure distribution}
\label{sec:uniformpressure}

For code verification, a one-dimensional problem was essentially solved in three dimensions.
Indeed, uniform pressure distributions were assumed to hold inside and outside, with values $p_\mathrm{in}=1\,\mathrm{MPa}$ and $p_\mathrm{out}=0\,\mathrm{MPa}$ respectively, so that they were independent of the azimuthal and axial directions, along with all the mechanics of the problem (i.e. $\pder{u}{\theta}=0$ and $\pder{u}{z}=0$).
The remaining boundary conditions at the axial faces also implied that $u_\theta=0$ and $u_z=0$.
Thus, the exact solution was derived from the ansatz that all nonvanishing physical variables were functions only of the radial direction, $r$.
With these assumptions, the linear elasticity problem with no external volumetric forces reduces to the scalar equation
\begin{equation}
\label{eq:1DLinElastReduction}
	\frac{1}{r}\frac{\dd}{\dd r}(r\sigma_{rr})- \frac{1}{r}\sigma_{\theta\theta} = 0\,,
\end{equation}
with boundary conditions, $(\sigma\!\cdot\!\hat{n}(R_\mathrm{in}))_r=-\sigma_{rr}(R_\mathrm{in})=p_\mathrm{in}$ and $(\sigma\!\cdot\!\hat{n}(R_\mathrm{out}))_r=\sigma_{rr}(R_\mathrm{out})=-p_\mathrm{out}$.
For the steel, the nonzero stress components are
\begin{equation}
	\begin{gathered}
 		\sigma_{rr}= (2\mu_{\mathrm{S}}+\lambda_{\mathrm{S}})\frac{\dd u_r}{\dd r} + \lambda_{\mathrm{S}}\frac{u_r}{r}\,,\qquad
 		\sigma_{\theta\theta}= (2\mu_{\mathrm{S}}+\lambda_{\mathrm{S}})\frac{u_r}{r} + \lambda_{\mathrm{S}}\frac{\dd u_r}{\dd r}\,,\\
 		\sigma_{zz}= \lambda_{\mathrm{S}}\Big(\frac{\dd u_r}{\dd r} + \frac{u_r}{r}\Big)\,,
 	\end{gathered}
\end{equation}
while for the rubber there is the additional incompressibility equation $\div(u)=\tder{u_r}{r}+\frac{u_r}{r}=0$ and the stress components are
\begin{equation}
	\begin{gathered}
 		\sigma_{rr}= 2\mu_{\mathrm{R}}\frac{\dd u_r}{\dd r}-p_0\,,\qquad
 		\sigma_{\theta\theta}=2\mu_{\mathrm{R}}\frac{u_r}{r}-p_0\,,\\
 		\sigma_{zz}= -p_0\,,
 	\end{gathered}
\end{equation}
for some constant, $p_0\in\R$.

This boundary value problem has the general solution
\begin{equation}
	u_r(r)=
		\begin{cases}
			Ar^{-1}&\quad\text{if }\,R_{\mathrm{in}}\leq r\leq R_{\mathrm{mid}}\,,\\
			Br+Cr^{-1}&\quad\text{if }\,R_{\mathrm{mid}}\leq r\leq R_{\mathrm{out}}\,,
		\end{cases}
\end{equation}
where the range $R_{\mathrm{in}}\leq r\leq R_{\mathrm{mid}}$ represents the rubber and the range $R_{\mathrm{mid}}\leq r\leq R_{\mathrm{out}}$ represents the steel.
Upon matching displacements and tractions at the interface, $R_{\mathrm{mid}}$, and applying the boundary conditions, the constants $A$, $B$, $C$ and $p_0$ in their general form are determined to be, 
\begin{equation}
\begin{aligned}
 	A &= \frac{1}{2d}
		\Big(-p_\mathrm{in}\big(\mu_\mathrm{S}R_\mathrm{mid}^2+(\lambda_\mathrm{S}+\mu_\mathrm{S})R_\mathrm{out}^2\big)
			+p_\mathrm{out}(\lambda_\mathrm{S}+2\mu_\mathrm{S})R_\mathrm{out}^2\Big)R_\mathrm{mid}^2 R_\mathrm{in}^2\,,\\
 	B &= \frac{1}{2d} 
 		\Big(-p_\mathrm{in}\mu_\mathrm{S}R_\mathrm{in}^2R_\mathrm{mid}^2-
 			p_\mathrm{out}\big((\mu_\mathrm{R}-\mu_\mathrm{S})R_\mathrm{in}^2-\mu_\mathrm{R}R_\mathrm{mid}^2\big)R_\mathrm{out}^2\Big)\,,\\
 	C &= \frac{1}{2d} \Big(-p_\mathrm{in}(\lambda_\mathrm{S}+\mu_\mathrm{S}) R_\mathrm{in}^2 
 		+p_\mathrm{out}\big((\lambda_\mathrm{S}+\mu_\mathrm{R}+\mu_\mathrm{S})R_\mathrm{in}^2
 			-\mu_\mathrm{R}R_\mathrm{mid}^2\big)\Big)R_\mathrm{mid}^2 R_\mathrm{out}^2\,,\\
	p_0 &= \frac{1}{d}\Big(p_\mathrm{in}\big((\mu_\mathrm{R}-\mu_\mathrm{S})(\lambda_\mathrm{S}+\mu_\mathrm{S}) R_\mathrm{out}^2
 			+\mu_\mathrm{S}(\lambda_\mathrm{S}+\mu_\mathrm{R}+\mu_\mathrm{S})R_\mathrm{mid}^2\big)R_\mathrm{in}^2\\
 				&\qquad\qquad\qquad\qquad\qquad\qquad\qquad\qquad\qquad
 					-p_\mathrm{out}\mu_\mathrm{R}(\lambda_\mathrm{S}+2\mu_\mathrm{S})R_\mathrm{out}^2R_\mathrm{mid}^2\Big)\,,\\
 	d &= \big((\mu_\mathrm{R}-\mu_\mathrm{S})(\lambda_\mathrm{S}+\mu_\mathrm{S})R_\mathrm{out}^2
 		+\mu_\mathrm{S}(\lambda_\mathrm{S}+\mu_\mathrm{R}+\mu_\mathrm{S})R_\mathrm{mid}^2\big)R_\mathrm{in}^2\\
 			&\qquad\qquad\qquad\qquad\qquad\qquad
 				-\mu_\mathrm{R}\big(\mu_\mathrm{S} R_\mathrm{mid}^2
 					+(\lambda_\mathrm{S}+\mu_\mathrm{S})R_\mathrm{out}^2\big)R_\mathrm{mid}^2\,.
\end{aligned}
\end{equation}

\vspace{-0.25cm}

\begin{figure}[!ht]
 	\centering
  \includegraphics[trim=0cm 0cm 0cm 0cm,clip=true,scale=0.39]{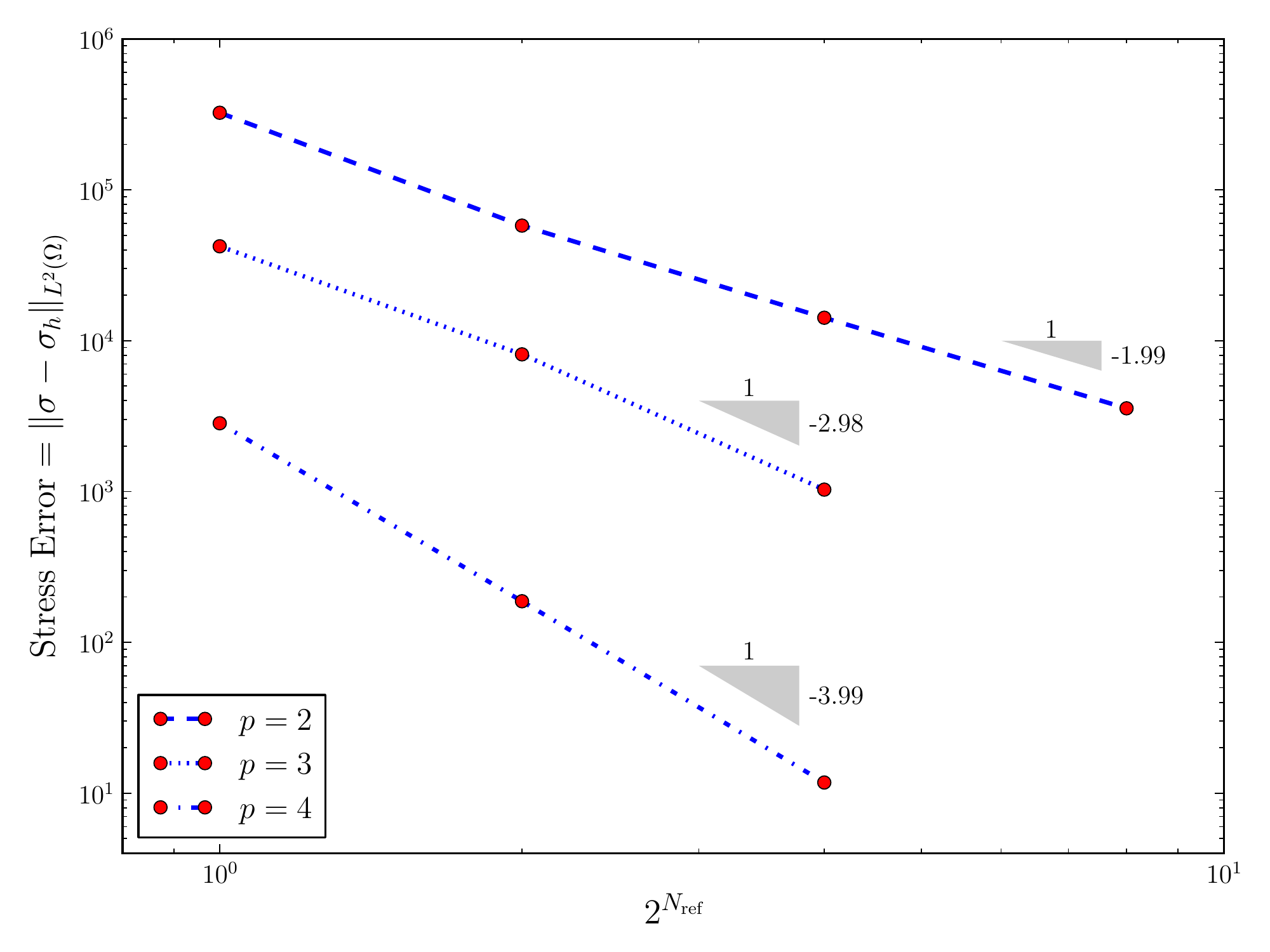}
  \caption{Stress error (in $\mathrm{Pa}$) as a function of the number of uniform refinements, $N_{\mathrm{ref}}$. The value of $p=1$ was not shown because the approximating isoparametric geometry was too inaccurate for the initial meshes.}
  \label{fig:high-material-contrast-uniform-error}
\end{figure}

\vspace{0.25cm}
Instead of showing convergence of the displacement error as in the previous example, here the convergence of the stress was presented.
For this, the $L^2(\Omega)$ error of the variable $\sigma_h$ was reported, where $\sigma_h$ is the $L^2(\Omega^\scU)$ ultraweak stress solution variable inside the rubber and $\sigma_h=\sfC\!:\!\nabla u_h$ inside the steel, with $u_h$ being the $H^1(\Omega^\scP)$ primal displacement solution variable.
Order $p$ convergence rates were expected for order $p$ discretizations as stipulated by Remark~\ref{rmk:ExactSequence}, because $\|\sigma-\sigma_h\|_{L^2(\Omega)}\leq\|\sigma-\sigma_h\|_{L^2(\Omega^\scU)}+\|\sfC\|\|u-u_h\|_{H^1(\Omega^\scP)}$.
This was corroborated numerically under uniform refinements for $2\leq p\leq 4$ as observed in Figure~\ref{fig:high-material-contrast-uniform-error}.


\subsubsection{Nonuniform pressure distribution}
\label{sec:nonuniformpressure}


Lastly, a nonuniform internal pressure distribution of $p_{\mathrm{in}}(\theta)=\cos^2(\theta)\,\mathrm{MPa}$ was prescribed on the inside, while the the external pressure was uniformly kept at $p_\mathrm{out}=0\,\mathrm{MPa}$.
After a few uniform refinements the solution is displayed in Figure~\ref{fig:NonuniformPressure} in each separate layer.
Note that the discontinuity of the stress component, $\sigma_{\theta\theta}$, which is useful in some applications, was amicably reproduced.

\begin{figure}[!ht]
 	\centering
  \includegraphics[trim=0cm 0cm 0cm 0cm,clip=true,scale=0.85]{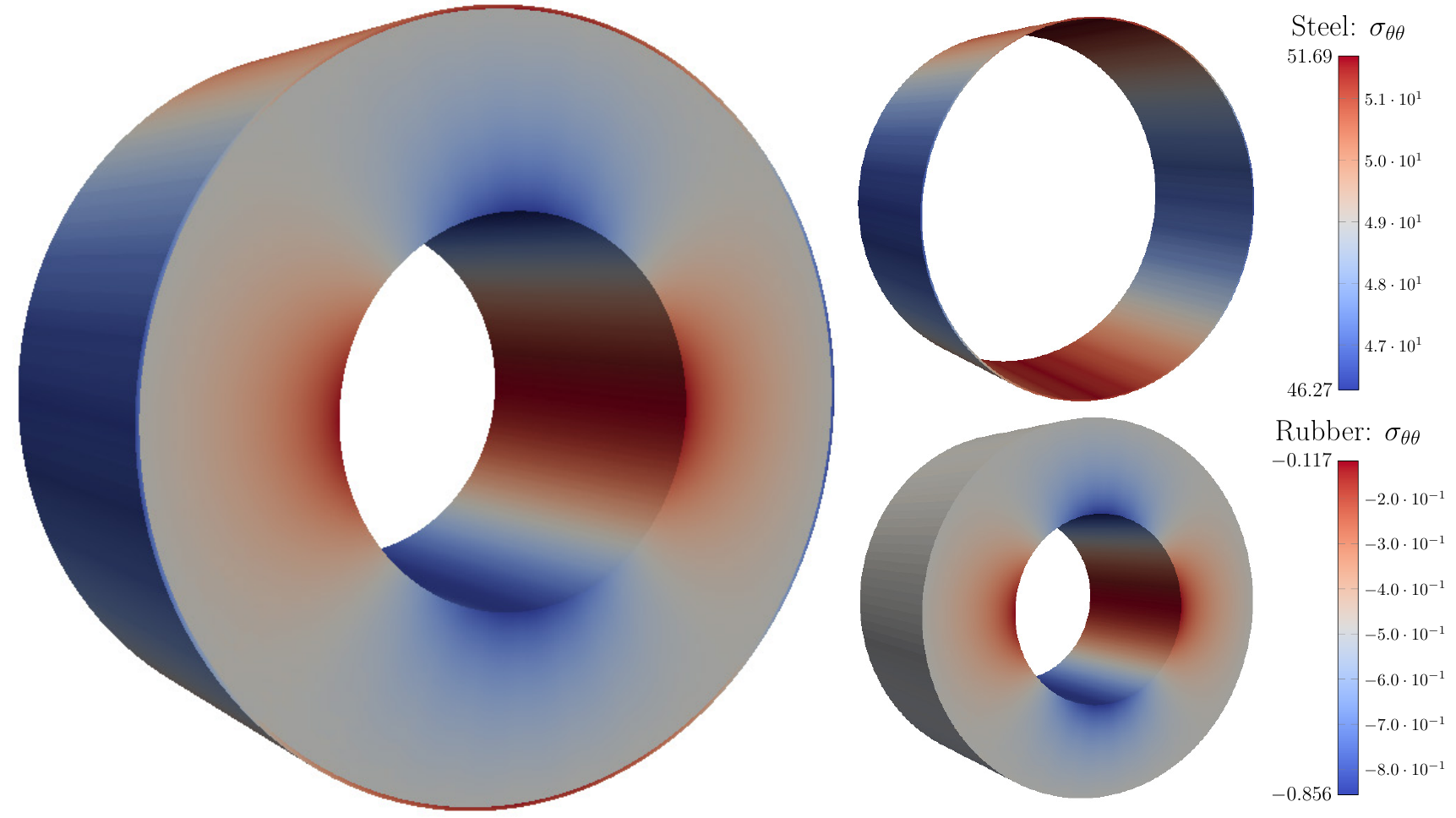}
 	\caption{Stress component $\sigma_{\theta\theta}$ (in $\mathrm{MPa}$) from computed solution with $p = 2$ and nonuniform internal pressure loading after three uniform refinements of the eight-element initial mesh. Note the discontinuity across the material interface.}
 	\label{fig:NonuniformPressure}
\end{figure}

\subsection{Discussion}
\label{sec:discussion}

An important point to emphasize about solving coupled formulations with the DPG methodology is that depending on the subdomain formulation, the mode of convergence, which is given by the minimization of the residual (see \eqref{eq:MinResidual}), is different because the residual is directly related to the bilinear form along with its associated trial and broken test spaces and their corresponding norms.
Hence, there is a potential subdomain bias in the convergence of the solution variables depending on the variational formulations associated to each subdomain.
This bias is even more stark when having a multi-material domain, which expectedly introduces different scales throughout the domain.
If the focus is to be centered around optimality in solution estimation, as research in the DPG methodology often has \cite{chan2014robust,demkowicz2011class,demkowicz2013robust}, then this subdomain bias in the convergence is a crucial aspect to ponder.
Indeed, it directly affects the (local) residual that is used as an a posteriori error estimator to drive adaptivity (see \eqref{eq:DiscreteResidual}).
The bias itself is not necessarily undesirable, since it may align with preferences by the end user, but it may be important to understand and modify so that it further aligns with those preferences.
For example, in the sheathed hose one might want to prioritize the values of stress in the steel over those in the rubber, or the values of strain in the rubber over those in the steel.
Thus, the ideal scenario is to be able to control the bias in accordance with a desired objective.
With this in mind, the lucid approach is to attempt to design \textit{objective-biased} test norms that satisfy this very purpose.
In this work, the use of the standard test norms in \eqref{eq:BrokenSpacesAndNorms} introduced a natural yet unoptimized subdomain bias, but designing these objective-biased norms would most certainly constitute an improvement to the scope of the present work.

Another remark to make is that the method used is akin to domain-decomposition methods and may even serve as an alternative.
In fact, many steps of the current method can be made parallel too, and the resulting connectivities along the interdomain boundaries are of a similar nature as those of other domain-decomposition methods.
Conceptually, compared to domain-decomposition methods, the use of the DPG methodology in the current method has the advantage of providing a solid ground of theory which practically guarantees stability and convergence of the solution with successive refinements, but has the disadvantage of possibly coming at a slightly higher computational cost.
In the future, it would be interesting to more rigorously investigate these connections with domain-decomposition methods.
Lastly, the interface variables in the broken formulations are not only natural to couple distinct formulations, but also suggest a natural way to couple with entirely different finite element methods or with boundary element methods.
{
The latter has already been further investigated in the context of elliptic transmission problems \cite{HeuerFuhrerCoupling,fuhrer2017coupling}.
}
\section{Conclusions}
\label{sec:Conclusion}

Coupled variational formulations for linear elasticity were constructed using a closely related family of broken variational formulations first presented in \cite{Keith2016}, where each subdomain of a partitioned domain was solved with a distinct formulation from this family.
The broken variational formulations that are commonplace in the context of the DPG methodology proved to be ideal in the theory and practice of the coupled formulations due to the presence of interface variables which served as a perfect vessel to transmit the solution information along the shared interdomain boundaries.
Indeed, the coupled formulations were proved to be well-posed and were successfully implemented and solved using the DPG methodology.
Expected convergence rates for various values of $p$ were observed for different variables in several well-crafted examples.

The coupled formulations are useful in cases where one might want to exploit the properties of a particular formulation in a certain part of the domain.
For example, it is useful to have a robust formulation when a part of the domain is composed of a nearly incompressible material, and certain formulations are more convenient when singular behavior of the solution is expected.
In this work, an example of a sheathed hose with high material contrast was used to illustrate the former point.
This included the derivation of a nontrivial and physically-relevant exact solution which can be used as a benchmark by other researchers.
Regarding the near singular behavior in the latter point, it would be interesting to study some examples with Maxwell's equations in the future.
The mode of convergence is different depending on which variational formulation is being used, so to obtain a desired convergence behavior or to have a better control of a residual-based adaptivity strategy, it would be compelling to explore more exotic test norms eventually.
Additionally, from both a theoretical and practical standpoint, the approach presented in this work can be extended with the help of the existing literature to many other equations such as Poisson's equation, Maxwell's equations \cite{BrokenForms15}, the diffusion-convection-reaction equation \cite{DemkowiczClosedRange}, and more.
It should also be noted that the implementation of these coupled formulations can be useful as an alternative to other domain-decomposition methods.

Finally, the DPG methodology was fundamental in achieving a successful implementation, in large part because of its versatility in displaying stability with a variety of distinct variational formulations, including those that have different trial and test spaces.
In this work, a novel and particularly productive decoupling of the stiffness matrix and residual computation was described (see \eqref{eq:CholeskyStiffness} and \eqref{eq:DiscreteResidual}), allowing for a considerably more efficient computational implementation of some of the linear algebra operations that accompany the DPG methodology.




\paragraph{Acknowledgements.}
This work was partially supported with grants by NSF (DMS-1418822),  AFOSR (FA9550-12-1-0484), and ONR (N00014-15-1-2496).
The authors would like to thank Dr.~Jean-Luc Cambier from AFOSR for raising the question on whether different DPG formulations could be applied in different subdomains and yield a globally coupled well-posed variational problem.


\phantomsection
\addcontentsline{toc}{section}{References}
\bibliographystyle{apalike}
\bibliography{main}


\end{document}